\documentclass[a4paper]{article}

\usepackage[utf8]{inputenc}
\usepackage{amsfonts}
\usepackage{amsmath}
\usepackage{amsthm}
\usepackage{psfrag}
\usepackage{enumitem}
\usepackage[]{units}

\usepackage{geometry}
\usepackage{tikz}
\usepackage{pgfplots}

\pgfplotsset{compat=newest}
\usetikzlibrary{calc}
\usetikzlibrary{patterns}

\usepackage{todonotes}
\usepackage{bbm}
\usepackage{bm}
\usepackage{xcolor}
\usepackage{floatrow}

\newcommand{\bp}{\mathbf p}
\newcommand{\bb}{\mathbf b}
\newcommand{\bx}{\mathbf x}
\newcommand{\by}{\mathbf y}
\newcommand{\bz}{\mathbf z}
\newcommand{\be}{\mathbf e}
\newcommand{\bq}{\mathbf q}
\newcommand{\bg}{\mathbf g}
\newcommand{\boldm}{\mathbf m}

\newcommand{\kl}[1]{\left(#1\right)}
\newcommand{\berr}{\bm{\eta}}
\newcommand{\edot}{\,\cdot\,}
\newcommand{\set}[1]{\left\{#1\right\}}
\newcommand{\norm}[1]{\left\Vert#1\right\Vert}
\newcommand{\abs}[1]{\left\vert#1\right\vert}
\newcommand\la{\lambda}

\newcommand\eps{\epsilon}

\newcommand\rmd{\mathrm d}

\newcommand{\N}{\mathbb{N}}
\newcommand{\R}{\mathbb{R}}
\newcommand{\C}{\mathbb{C}}
\newcommand{\bbS}{\mathbb{S}}

\newcommand{\ee}{\mathrm{e}}

\newcommand{\dx}{\mathrm{d}}

\DeclareMathOperator{\Mo}{\mathbf M}
\DeclareMathOperator{\Ho}{\mathbf H}

\newcommand{\Ncal}{\mathcal{N}}

\DeclareMathOperator{\Psio}{\boldsymbol\Psi}
\DeclareMathOperator*{\minimize}{minimize}
\DeclareMathOperator*{\suchthat}{such\,that}
\newcommand{\To}{\mathbf{T}}
\newcommand{\Ao}{\mathbf{A}}

\newcommand{\Ocal}{\mathcal{O}}
\newcommand{\supp}{\operatorname{supp}}

\newtheorem{theorem}{Theorem}
\theoremstyle{definition}
\newtheorem{definition}[theorem]{Definition}

\numberwithin{theorem}{section}
\numberwithin{figure}{section}

\title{A Novel Compressed Sensing Scheme for Photoacoustic Tomography}

\date{January 2015}

\author{M. Sandbichler\footnotemark[2]
\and F. Krahmer\footnotemark[3]
\and T.  Berer\footnotemark[4]
\and P. Burgholzer\footnotemark[4]
\and M. Haltmeier\footnotemark[2]}

\footnotetext[2]{Department of Mathematics, University of Innsbruck,
Technikerstra{\ss}e 13a, A-6020 Innsbruck, Austria. E-mail:
{\tt \{michael.sandbichler,markus.haltmeier\}@uibk.ac.at}}
\footnotetext[3]{Institute for Numerical and Applied Mathematics, University of G\"ottingen,
Lotzestra{\ss}e 16-18, 37083 G\"ottingen, Germany. E-mail:
{\tt f.krahmer@math.uni-goettingen.de}}
\footnotetext[4]{Christian Doppler Laboratory for Photoacoustic Imaging and Laser Ultrasonics,
and Research Center for Non-Destructive Testing (RECENDT), Altenberger Stra{\ss}e 69, 4040 Linz, Austria. E-mail: {\tt \{thomas.berer,peter.burholzer\}@recendt.at}}

\begin{document}
\maketitle

\begin{abstract}
Speeding up the data acquisition is one of the central aims to advance  tomographic imaging. On the one hand, this reduces motion artifacts due to undesired movements, and on the other hand this decreases the examination time for the patient. In this article, we propose a new scheme for speeding  up the data collection process in photoacoustic tomography.
Our proposal is based on compressed sensing and reduces  acquisition time and system costs while maintaining image quality.
As measurement data we  use random combinations of pressure values that we use to recover a complete set of pressure data prior to the actual image reconstruction.
We obtain theoretical recovery guarantees for our compressed sensing scheme and support the theory by reconstruction results on simulated data as well as on experimental data.


\bigskip\noindent{\bf AMS classification numbers.}
45Q05, 94A08, 92C55.
\end{abstract}

\section{Introduction}
\label{sec:intro}

Photoacoustic tomography (PAT) is a recently developed non-invasive medical imaging technology whose benefits combine the high contrast of
pure optical imaging with the high spatial resolution of pure ultrasound imaging \cite{Bea11,Wan09b,XuWan06}.
In order to speed up the measurement process, in this paper we propose a novel compressed sensing approach for PAT that
uses random combinations of the induced pressure as measurement data.
The proposed strategy yields recovery guarantees and furthermore comes with an efficient numerical implementation allowing high resolution real time  imaging.
We thereby focus on a variant of PAT using integrating line detectors proposed in \cite{BurHofPalHalSch05,PalNusHalBur07a,BerEtAl12}.
Our strategy, however,  can easily be  adapted  to  more classical PAT setups using arrays of point-like detectors.

\begin{figure}[tbh!]
\begin{center}
    \includegraphics[width=\textwidth]{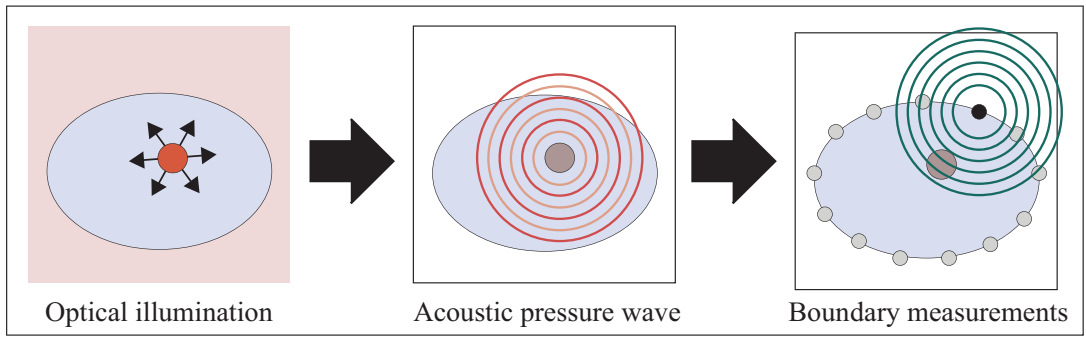}
\end{center}
\caption{\textsc{Basic\label{fig:pat} principle of PAT}. A semi-transparent sample is illuminated with a  short optical pulse that induces an acoustic pressure wave.
The induced  pressure  is measured outside of the sample and used to recover an image of the interior.}
\end{figure}

Our proposal is based on the main components of  compressed sensing,
namely randomness and sparsity.
Compressed sensing is one of the most influential discoveries in applied mathematics and signal processing of the past decade \cite{candes2006robust,donoho2006compressed}.
By combining  the  benefits of data compression and data acquisition
it allows to recover a signal from far fewer linear measurements than suggested by Shannon's sampling theorem.
It has led to several new proposed sampling strategies in medical imaging, for example for  speeding up MRI data acquisition (see \cite{lustig2007sparse, lustig2008compressed}) or completing under-sampled CT images~\cite{chen2008prior}. Another prominent application of compressed sensing is the single pixel camera  (see \cite{DuaetAl08}) that
circumvents the use of several expensive high resolution sensors in digital photography.

\subsection{Photoacoustic tomography (PAT)}

PAT is based on the generation of acoustic waves by illuminating a semi-transparent  sample  with short optical pulses (see Figure \ref{fig:pat}).
When the sample is illuminated with a short laser pulse, parts of the  optical energy become absorbed.
Due to thermal expansion a subsequent  pressure wave  is generated depending on the structure  of the sample.
The induced  pressure waves  are recorded outside of the  sample and used to recover an image of the interior, see \cite{BurBauGruHalPal07,KucKun08,XuWan06,Wan09b}.

\begin{figure}[htb!]
\floatbox[{\capbeside\thisfloatsetup{capbesideposition={right,bottom},capbesidewidth=0.45\textwidth}}]{figure}[\FBwidth]
{\caption{\textsc{PAT with integrating line detectors.}
An array of integrating line detectors measures  integrals of the induced acoustic
pressure over a certain number of  parallel lines. These data are used to recover a linear projection of the object in the first step.
By rotating the array of line detectors around a single axis several projection  images are obtained and used to recover the actual three dimensional object in a second step.}
\label{fig:IntegratingPAT}}
{ \includegraphics[width=0.35\textwidth]{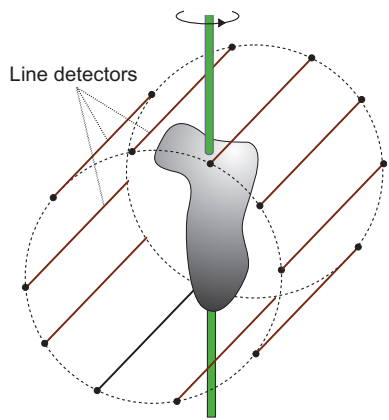}\quad}
\end{figure}

In this paper, we consider a special variant of photoacoustic tomography that uses integrating line detectors for recording the pressure waves, as proposed in \cite{BurHofPalHalSch05}.
As illustrated in Figure \ref{fig:IntegratingPAT}, an array of line detectors is arranged around the investigated sample and  measures integrals of the
pressure wave  over a certain number of parallel lines.
Assuming  constant speed of sound, the pressure integrated along the direction  of the line detectors satisfies the  two dimensional wave equation
\begin{equation}  \label{eq:wave}
	\left\{ \begin{aligned}
	 \partial^2_t p (x,t)  - \Delta p(x,t)
	&=
	0 \,,
	 && \text{ for }
	\kl{x,t} \in
	\R^2 \times \kl{0, \infty}
	\\
	p\kl{x,0}
	&=
	f(x) \,,
	&& \text{ for }
	x  \in \R^2
	\\
	\partial_t
	p\kl{x,0}
	&=0 \,,
	&& \text{ for }
	x  \in \R^2 \,,
\end{aligned} \right.
\end{equation}
where the time scaling is chosen in such a way that  the speed of sound is normalized to one.
The initial datum $f$ in  \eqref{eq:wave} is  the two dimensional projection image
of the actual, three dimensional initial pressure distribution.

Image reconstruction  in PAT with integrating line detectors can be performed via a two-stage approach \cite{BurBauGruHalPal07,PalNusHalBur07b}.
In the first step the measured pressure values corresponding to values of the solution of \eqref{eq:wave} outside the support of $f$ are used to reconstruct a  linear projection (the initial data in \eqref{eq:wave}) of the three dimensional initial pressure distribution. This procedure is repeated by rotating the array of  line detectors around a single axis which yields  projection images for several angles. In a second step these projection images are used to recover the  three dimensional initial pressure distribution by  inverting  the classical Radon transform.
In this work we  focus on the first problem of reconstructing the (two dimensional)
initial pressure distribution $f$ in \eqref{eq:wave}.

Suppose that the integrating line detectors are arranged on the surface of a  circular cylinder of radius $\rho>0$, and that the object is located inside that cylinder (see Figure \ref{fig:IntegratingPAT}).
The data measured by the array of line detectors are then modeled by
\begin{align} \label{eq:data}
    & p_j := p\kl{ z_j, \edot}
    \colon [0,2\rho]  \to \R   \,,
    \\ \label{eq:sampling}
	&
	z_j
	:=  \begin{pmatrix} \rho \cos \kl{2\pi(j-1)/N} \\
	 \rho\sin \kl{2\pi(j-1)/N}  \end{pmatrix}\,,
\quad \text{ for }  j = 1, \dots,  N   \,,
\end{align}
where $p_j $ is the pressure signal corresponding to the $j$-th line detector.
Since the two dimensional initial pressure  distribution $f$ is supported in a disc of radius $\rho$
and the speed of sound is constant and normalized to one, no additional information is  contained in data
$p\kl{ z_j, t}$ for $t > 2\rho$. This can be seen, for example by exploiting the
explicit relations between the two dimensional pressure signals $p\kl{ z_j, \edot}$
and  the spherical means (compare with Subsection~\ref{sec:sparsifying})
of the initial pressure distribution; see \cite{PalNusHalBur09}.

Of course, in practice also the functions $p_j \colon [0,2\rho]  \to \R$ have to be
represented by discrete  samples.  However, temporal samples can easily be collected at a high sampling rate compared to the spatial sampling, where each sample requires a separate sensor. It is therefore natural to consider  the semi-discrete data model \eqref{eq:data}.
Our compressed PAT scheme could easily be adapted to a fully discretized data model.

\subsection{Compressed sensing PAT}

When using data of the form~\eqref{eq:data}, high resolution imaging requires the number $N$ of detector locations to be sufficiently large. As the fabrication of an array of parallel line detectors is demanding, most experiments using  integrating line detectors have been carried out using a single line detector, scanned on circular paths using scanning stages~\cite{NusJBO10, GruJBO10}. Recently, systems using arrays of $64$ parallel line detectors have been demonstrated~\cite{GraPSP14, BauPSP15}. The most costly building blocks
in such devices are the analog to digital converters (ADC). For completely parallel readout, a separate ADC is required for every detector channel. In order to reduce costs, in these practical implementations two to four line detector channels are multiplexed to one ADC.  For collecting the complete pressure data, the measurements  have to be performed two (respectively four) times, because only 32 (respectively 16) of the 64 line measurements can be  read out in parallel.  This, again, leads to an increased overall measurement time. For example, using an excitation laser with a repetition rate of $\unit[10]{Hz}$ and two times multiplexing, a measurement time of
$\unit[0.2]{s}$ for a projection image has been reported in~\cite{GraPSP14}. Without multiplexing,
this  measurement time would reduce to the half.

In order to speed up the scanning process and to reduce system costs, in this paper we propose a novel compressed sensing approach that allows to perform a smaller number of random measurements with a reduced  number of ADCs, while retaining high spatial resolution.
For that purpose, instead of collecting
individually  sampled data $p_j(t)$ as in \eqref{eq:data}, we  use random combinations
\begin{equation} \label{eq:cs1}
	y_i(t)  =  \sum_{j \in J_i} p_j\kl{t}
	\quad \text{ for } i \in \set{ 1, \dots,  m} \text{ and } t \in  [0,2\rho] \,,
\end{equation}
where $m \ll N$ is the number of compressed sensing measurements and $J_i \subset \{ 1, \dots, N\}$ corresponds
to the random set of  detector locations contributing to the $i$-th measurement.
In the  reconstruction process  the  random linear combinations $y_i(t)$ are used to recover the full set of pressure data $ p_1(t), \dots, p_N(t)$ using compressed
sensing techniques. The initial pressure distribution $f$ in \eqref{eq:wave}  is
subsequently recovered from the completed pressure data by applying standard PAT reconstruction algorithms such as time reversal \cite{BurMatHalPal07,HriKucNgu08,Treeby10} or filtered backprojection \cite{BurBauGruHalPal07,finch2007inversion,FinPatRak04,Hal14,Kun07a,XuWan05}.

A naive approach for recovering the pressure data from the random measurements $y_i(t)$  would be to solve  \eqref{eq:cs1} for $p_j(t)$
separately for each $t\in [0,2\rho]$. Since $m \ll N$, this is a severely underdetermined system of linear equations and its solution requires appropriate prior knowledge of the unknown  parameter vector.  Compressed sensing suggests to use the sparsity of the parameter vector in a suitable basis  for that purpose.
However, recovery guarantees for zero/one measurements of the type \eqref{eq:cs1} are basis-dependent and  require  the parameter to be  sparse in the standard basis rather than sparsity in a different basis such as orthonormal wavelets (see Subsection \ref{sec:lossless}). However, for pressure signals \eqref{eq:data} of practical relevance such sparsity assumption in the original basis does not hold.

In this work we therefore propose a different approach for solving  \eqref{eq:cs1} by  exploiting special properties of  the data in PAT.
For that purpose we apply a transformation that acts in the temporal variable only, and makes the transformed pressure values  sufficiently  sparse in the spatial (angular) component.
In Subsection \ref{sec:sparsifying} we present an example of such a transform.
The application of  a sparsifying transform  to \eqref{eq:cs1} yields linear equations with unknowns being sparse in the angular variable.
It  therefore allows to apply sparse recovery results for the zero/one
measurements under consideration.

 \subsection{Relations to previous work}

A different compressed sensing approach for  PAT  has been considered in~\cite{provost2009application,guo2010compressed}. In these articles, standard point
samples  (such as  \eqref{eq:data}) have been used as measurement data and no  recovery guarantees have been derived.
Further, in  \cite{provost2009application,guo2010compressed} the phantom is directly reconstructed  from the incomplete data, whereas we first  complete the data using sparse recovery techniques.
Our approach is more related to  a compressed sensing approach for PAT using a planar detector array that has been proposed in~\cite{huynh2014patterned} and also uses random zero/one combinations of pressure values  and recovers the complete pressure prior to the actual image reconstruction. However, in \cite{huynh2014patterned} the sparsifying transform is applied in spatial domain where recovery guarantees are not available as noted above. We finally notice that our proposal of using a sparsifying temporal transform can easily be extended to planar detector arrays in two or three spatial dimensions; compare Section~\ref{sec:discussion}.

\subsection{Outline}

The rest of this paper is organized as follows. In Section~\ref{sec:background} we review  basic results from compressed sensing  that we  require for our proposal. We therefore focus on recovery guarantees for zero/one matrices modeled by lossless expanders; see Subsection \ref{sec:lossless}.  In Section~\ref{sec:cspat} we present the mathematical framework of the proposed PAT compressed sensing scheme. The   sparsity in the spatial variable, required for $\ell^1$-minimization, is obtained by applying a transformation acting in the temporal variable.       
An example of such a transformation is given in Subsection~\ref{sec:sparsifying}. In Section~\ref{sec:num} we present numerical results supporting our theoretical investigations.
The paper concludes  with a short discussion in
Section~\ref{sec:discussion}.

\section{Background from compressed sensing}
\label{sec:background}

In this section we shortly review basic concepts and results of compressed sensing
(sometimes also termed compressive sampling). Our main focus will be on recovery results
for lossless expanders, which are the basis of our PAT compressed
sensing scheme.

\subsection{Compressed sensing}\label{sec:cs}

Suppose one wants to sense a high dimensional data vector $\bx = (x_1, \dots, x_N )\in \R^N$,
such as a digital image. The classical sampling approach is to measure each component of $x_i$
individually. Hence, in order to collect the whole data vector one has to perform $N$ separate measurements,
which may be too costly.
On the other hand it is well known and the basis of data compression algorithms, that many datasets are compressible in a suitable basis. That is, a limited amount of information is sufficient to capture the high dimensional vector $\bx$.

Compressed sensing incorporates this compressibility observation into the sensing mechanism \cite{candes2006robust,CanTao06,donoho2006compressed}.
Instead of measuring  each coefficient of the data vector  individually, one collects linear measurements
\begin{equation}\label{eq:system1a}
	\Ao \bx   = \by \,,
\end{equation}
where $\Ao  \in \R^{m \times N}$ is the measurement matrix with $m \ll N$,
and $\by =(y_1, \dots, y_m) \in \R^m$ is the
measurement vector. Any component of the data vector can be interpreted as a scalar linear  measurement performed on the unknown $\bx$, and the assumption $m \ll N$ means that far  fewer measurements than parameters  are available.
As $m \ll N$, the system \eqref{eq:system1a} is highly   underdetermined and cannot be uniquely solved (at least without  additional information) by standard linear algebra.

Compressed sensing overcomes this obstacle by utilizing randomness and sparsity.
Recall that  the vector $\bx  = (x_1, \dots, x_N)$ is called $s$-sparse if the support
\begin{equation} \label{eq:sparsitya}
	\supp \kl{\bx} :=  \set{j \in \set{1, \dots , N}
	\colon x_j \neq 0 }
\end{equation}
contains at most $s$ elements.
Results from compressed sensing
state that for suitable $\Ao$, any $s$-sparse $\bx \in \R^n$ can be found
via the optimization problem
\begin{equation} \label{eq:CSell1}
\begin{aligned}
	&\minimize_{\bz \in \R^N}
	&&\norm{\bz}_1 = \sum_{j=1}^N \abs{z_j}
	\\
	&\suchthat
	&& \Ao \bz  =  \by   \,.
\end{aligned}
\end{equation}
By relaxing the equality constraint $\Ao \bz  =  \by$, the optimization problem \eqref{eq:CSell1}
can be adapted to data which are only approximately sparse and noisy \cite{CanRomTao06b}.

A sufficient condition to guarantee recovery is the so called \emph{restricted isometry property} (RIP), requiring that for any $s$-sparse vector $x$, we have
\begin{equation*}
(1-\delta)\Vert \bx\Vert_2\leq \Vert \Ao \bx \Vert_2 \leq (1+\delta) \Vert \bx\Vert_2
\quad \text{ for some small } \delta \in (0,1)\,.
\end{equation*}
The smallest constant $\delta$ satisfying this inequality is called the $s$-restricted isometry constant of $A$ and denoted by $\delta_s$. Under certain conditions on $\delta_s$, recovery guarantees for sparse and approximately sparse data can be obtained, see for example~\cite{candes2008restricted,cai2014sparse}.

While the restricted isometry itself is deterministic, to date all constructions that yield near-optimal embedding dimensions $m$ are based on random matrices.
Sub-gaussian random matrices satisfy  the RIP with high probability for an order-optimal embedding dimension $m=\Ocal(s\log(N/s))$, see e.g.~\cite{baraniuk2008simple}.
Partial random Fourier matrices (motivated by MRI measurements) and subsampled random convolutions (motivated by remote sensing) have been shown to allow for order-optimal embedding dimensions up to logarithmic factors, see~\cite{rudelson2008sparse, rauhut2007random} and~\cite{rarotr12, KMR12}, respectively.

The sparsity is often not present in the standard basis of $\R^N$, but in a special \emph{sparsifying} basis,
such as wavelets. For matrices with subgaussian rows this does not cause a problem, as the rotation invariance of subgaussian vectors ensures that after incorporating an orthogonal transform, the resulting random matrix construction still yields RIP matrices with high probability. As a consequence, the sparsifying basis need not be
known for  designing the measurement matrix $\Ao$, which is often referred to as
universality of such  measurements \cite{baraniuk2008simple}.

Many structured random measurement systems including the partial random Fourier and the subsampled random convolution scenarios mentioned above, however, do not exhibit universality. For example, one can easily see that subsampled Fourier measurements cannot suffice if the signal is sparse in the Fourier basis. While it has been shown that this problem can be overcome by randomizing the column signs \cite{krwa11}, such an alteration often cannot be implemented in the sensing setup.  Another way to address this issue is by requiring incoherence between the measurement basis and the sparsity basis \cite{CaRo07}. That is, one needs that inner products between vectors of the two bases are uniformly small. If not all, but most of these inner products are small, one can still recover, provided that one adjusts the sampling distribution accordingly; this scenario includes the case of Fourier measurements and Haar wavelet sparse signals \cite{KW13}.

Incoherence is also the key to recovery guarantees for gradient sparse signals. Namely, many natural images are observed to have an approximately sparse  discrete gradient. As a consequence, it has been argued using a commutation argument that one can recover the signal from uniformly subsampled Fourier measurements via minimizing the $\ell^1$ norm of the discrete gradient, the so-called total variation (TV) \cite{candes2006robust}.
TV minimization had already proven to be a favorable method in image processing, see, for example \cite{RudOshFat92, chambolle1997image}.
A problem with this approach is that the compressed sensing recovery guarantees then imply good recovery of the gradient, not of the signal itself. Small errors in the gradient, however, can correspond to substantial errors in the signal, which is why this approach can only work if no noise is present. A refined analysis that allows for noisy measurements requires the incoherence of the measurement basis to the Haar wavelet basis \cite{needell2013stable}. Again, TV minimization is considered for recovery. By adjusting the sampling distribution, these results have also been shown to extend to Fourier measurements and other systems with only most measurement vectors incoherent to the Haar basis \cite{KW13}.

For the measurement matrices considered in this work, namely  zero/one matrices
based on expander graphs, recovery guarantees build on an $\ell^1$-version of the restricted isometry property, namely one requires that
	\[(1-\delta)\Vert \bx\Vert_1\leq \Vert \Ao \bx \Vert_1 \leq (1+\delta) \Vert \bx\Vert_1\]
for all sufficiently sparse $\bx$ and some constant $\delta>0$;
see \cite{berinde2008combining} and Subsection~\ref{sec:lossless} below.
As the $\ell^1$-norm is not rotation invariant, basis transformations
typically destroy this property.
That is, not even incoherence based recovery guarantees are available; recovery results only hold in the standard basis. Thus an important aspect of our work will be to ensure (approximate) sparsity in the standard basis. This will be achieved by applying a particular transformation in the time variable.

\subsection{Recovery results for lossless expanders}
\label{sec:lossless}

Recall that we seek recovery guarantees for a measurement setup, where each detector is switched on exactly $d$ out of $m$ times. That is, one obtains a binary measurement matrix $\Ao \in \{0,1\}^{m\times N}$ with exactly $d$ ones in each column.
It therefore can be interpreted as the adjacency matrix of a left $d$-regular bipartite graph.
Under certain additional conditions, such a bipartite graph is a lossless expander (see Definition~\ref{def:expander}) which, as we will see, guarantees stable recovery of sparse vectors. Expander graphs have been used since the 1970s in theoretical computer science, originating in switching theory from modeling networks connecting many users, cf.~\cite{klawe1984limitations} for further applications. They have also been useful in measure theory, where it was possible to solve the Banach-Ruziewicz problem using tools from the construction of expander graphs, see~\cite{lubotzky1994discrete} for a detailed examination of this connection. For a survey on expander graphs and their applications, see for example~\cite{hoory2006expander, lubotzky2012expander}.

Compressed sensing with expander graphs has been considered in~\cite{berinde2008combining,berinde2008practical,indyk2008near,jafarpour2009efficient,xu2007efficient}, where also several efficient algorithms for the solution of compressed sensing problems using expander graphs have been proposed. A short review
of sparse recovery algorithms using expander like matrices is given in~\cite{gilbert2010sparse}.
In this subsection we recall main compressed sensing results using lossless expanders as presented in  the recent monograph \cite{foucart2013mathematical}, where the proofs of all mentioned theorems can be found in Section~13.

\begin{figure}[tbh!]
\begin{center}
    \includegraphics[width=0.4\textwidth]{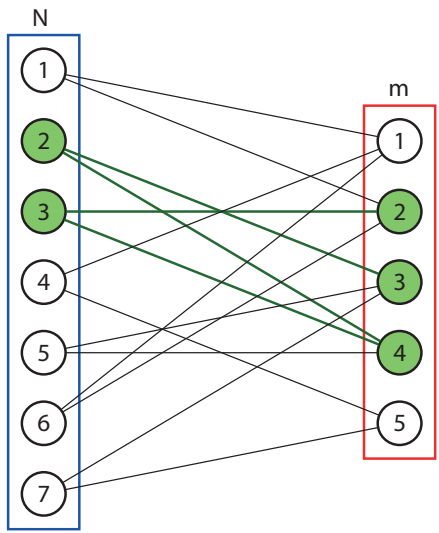}
\end{center}
\caption{Example\label{fig:regular} of a left $d$-regular bipartite graph with $d=2$, $N=7$ left vertices, and $m=5$ right vertices. Here $d=2$ because exactly $2$ edges emerge at each left vertex. For $J=\set{2,3}$ the set of right vertices connected to $J$ is given by $R(J) = \set{2,3,4}$.}
\end{figure}

Recall that a  bipartite graph consists of a triple $(L,R,E)$,  where $L$ is the set of left vertices, $R$ the set of right vertices, and $E \subset L \times R$ is the set of edges.  Any element  $(j,i) \in E$ represents an edge with left vertex $j\in L$ and right vertex $i \in R$.

A bipartite graph $(L,R,E)$ is called $d$-regular for some $d \geq 0$,  if for every given left vertex $j\in L$, the number of edges $(j,i) \in E$ emerging from  $j$ is exactly equal to $d$.  Finally, for any subset $J \subset  L$ of left vertices, let
\begin{equation*}
	R(J) :=
	\set{ i \in R \colon
	 \text{ there exists some } j \in J
	 \text{ with }  (j, i)\in E
	  }
\end{equation*}
denote the set of right vertices connected to $J$.
Obviously,  for any left $d$-regular bipartite graph and any $J \subset  L$, we have  $|R(J)|\leq d|J|$.

\begin{definition}[Lossless expander]\label{def:expander}
Let $(L, R, E)$ be a left $d$-regular bipartite graph,   $s \in \N$, and $\theta \in [0,1]$.
Then $(L, R, E)$ is called an $\kl{s,d,\theta}$-lossless expander, if
\begin{equation}\label{eq:expander}
	|R(J)|
	\geq
	(1-\theta) d |J|
	\quad \text{ for all } J\subset L  \text{ with }
	|J| \leq s  \,.
\end{equation}
For any left $d$-regular bipartite graph, the smallest number $\theta\geq 0$ satisfying \eqref{eq:expander} is called the $s$-th  restricted expansion constant and denoted by $\theta_s$.
\end{definition}

The following theorem states that a randomly chosen $d$-regular bipartite graph will be
a lossless expander with high probability.

\begin{theorem}[Regular bipartite graphs are expanders with high probability] \label{thm:expander}
For every $0 < \eps < \frac12$, every $\theta \in (0,1)$ and every  $s \in \N$, the  proportion of $(s,d,\theta)$-lossless expanders
among  the set of all left $d$-regular bipartite graphs  having $N$ left vertices and $m$ right vertices exceeds
$1- \eps$, provided that
\begin{equation} \label{eq:expanderchoice}
d = \left\lceil{\frac1\theta \ln\left(\frac{\ee N}{\eps s}\right)}\right\rceil\quad\text{and}\quad
m\geq c_{\theta} s \ln\left(\frac{\ee N}{\eps s}\right) \,.
\end{equation}
Here $c_{\theta}$ is a constant only depending  on $\theta$, $\ee$ is Euler's constant,
$\ln(\edot) $ denotes  the natural logarithm and $\lceil x\rceil$ denotes the smallest integer larger or equal to $x$.
\end{theorem}

According to Theorem \ref{thm:expander}, any randomly chosen left $d$-regular bipartite graph is a lossless expander with high probability, provided that \eqref{eq:expanderchoice} is satisfied.
The following theorem states  that
the adjacency matrix $\Ao \in \set{0,1}^{m\times N}$,
\begin{equation} \label{eq:am}
	\Ao_{ij}
	= 1
	:\iff
	(j,i) \in E \,,
\end{equation}
of any lossless expander  with left vertices $L = \set{1, \dots, N}$ and right vertices $R = \set{1, \dots, m}$ yields stable recovery of
any sufficiently sparse parameter vector. The result was first established in~\cite{berinde2008combining}, we present the version found in~\cite{foucart2013mathematical}.

\begin{theorem}[Recovery guarantee for lossless expanders] \label{thm:stable}
Let $\Ao\in\{ 0, 1\}^{m\times N}$ be the adjacency matrix of a left $d$-regular bipartite graph having  $\theta_{2s}<1/6$.
Further, let $\bx\in\C^N$, $\eta >0$ and $\be\in\C^m$ satisfy $\Vert \be\Vert_1\leq\eta$, set $\bb:= \Ao\bx +\be$, and denote by $\bx_\star$ a solution of
\begin{equation} \label{eq:bp}
\begin{aligned}
&\minimize_{\bz \in \C^N} &&\norm{\bz}_1\\
&\suchthat &&\norm{\Ao \bz-\bb}_1
\leq \eta  \,.
\end{aligned}
\end{equation}
Then
\[
\Vert \bx - \bx_\star\Vert_1\leq \frac{2(1-2\theta_{2s})}{(1-6\theta_{2s})}\sigma_s(\bx)_1 + \frac{4}{(1-6\theta_{2s})d}\eta \,.\]
Here  the quantity $\sigma_s(\bx)_1:=\inf\{\Vert \bx - \bz\Vert_1\colon \bz\text{ is }s\text{-sparse}\}$ measures by how much the vector  $\bx \in \C^N$ fails to be $s$-sparse.
\end{theorem}

Combining Theorems \ref{thm:expander} and \ref{thm:stable}, we can conclude that the adjacency matrix $\Ao$ of a randomly chosen
left $d$-regular bipartite graph will, with high probability,  recover  any  sufficiently sparse vector $\bx \in \C^N$ by basis pursuit reconstruction \eqref{eq:bp}.

\section{Mathematical framework of the proposed   PAT compressed sensing scheme}
\label{sec:cspat}

In this section we describe our proposed compressed sensing strategy. As mentioned in the introduction, we
focus on PAT with integrating line detectors,  which is  governed by the two dimensional wave equation \eqref{eq:wave}.
In the following we first describe the compressed sensing measurement setup in Subsection~\ref{sec:setup}
and describe the sparse recovery strategy in Subsection~\ref{sec:recon}.
As the used pressure data are not sparse in the original domain we  introduce a temporal transform that  makes the data sparse in the spatial domain.
In Subsection~\ref{sec:sparsifying} we present an example of such a sparsifying
temporal transform.
In Subsection~\ref{sec:summary}  we finally summarize the whole PAT compressed sensing scheme.

\begin{figure}[tbh]
\centering \includegraphics[width=0.4\textwidth]{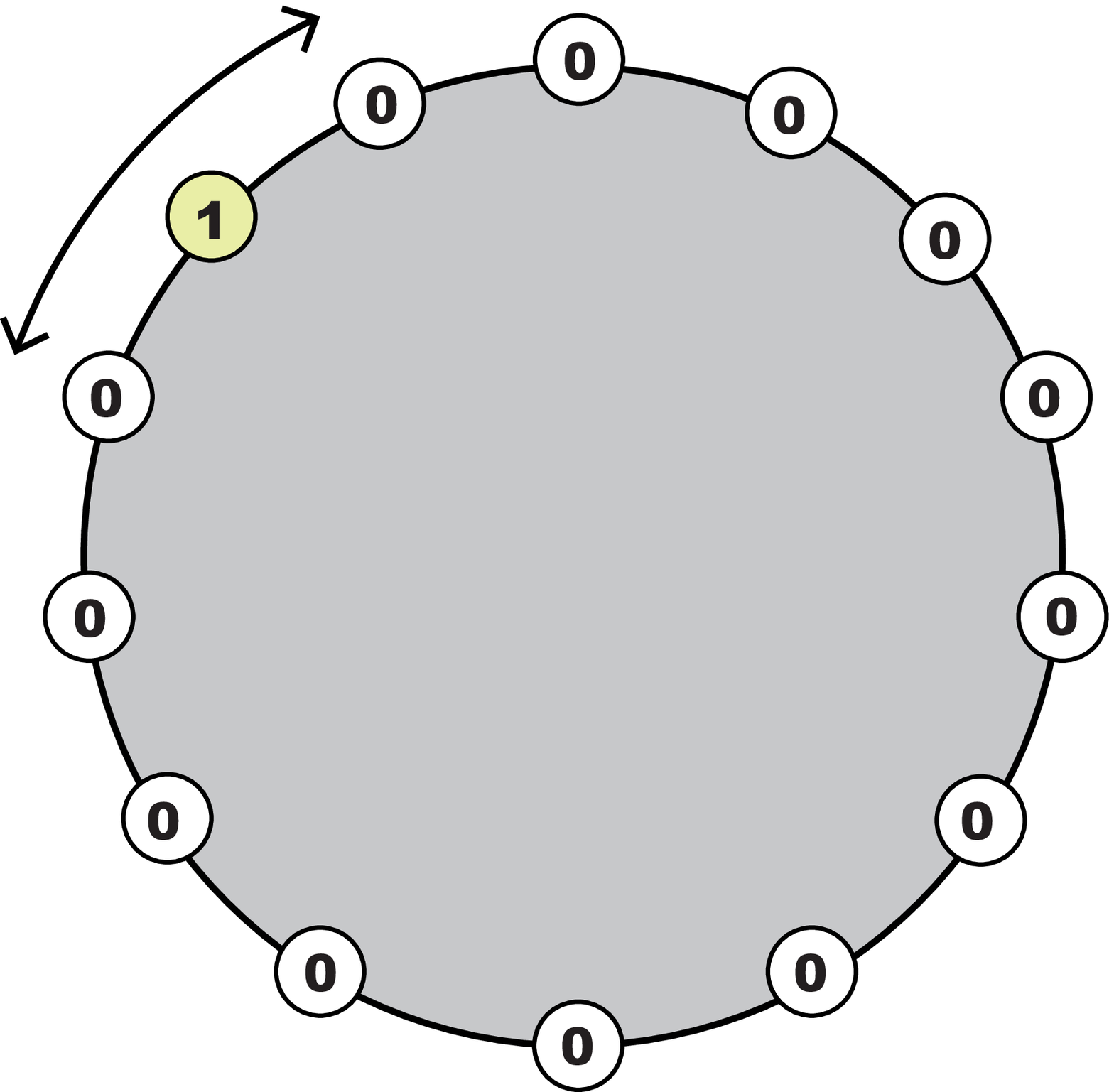}\hspace{0.08\textwidth}
\includegraphics[width=0.4\textwidth]{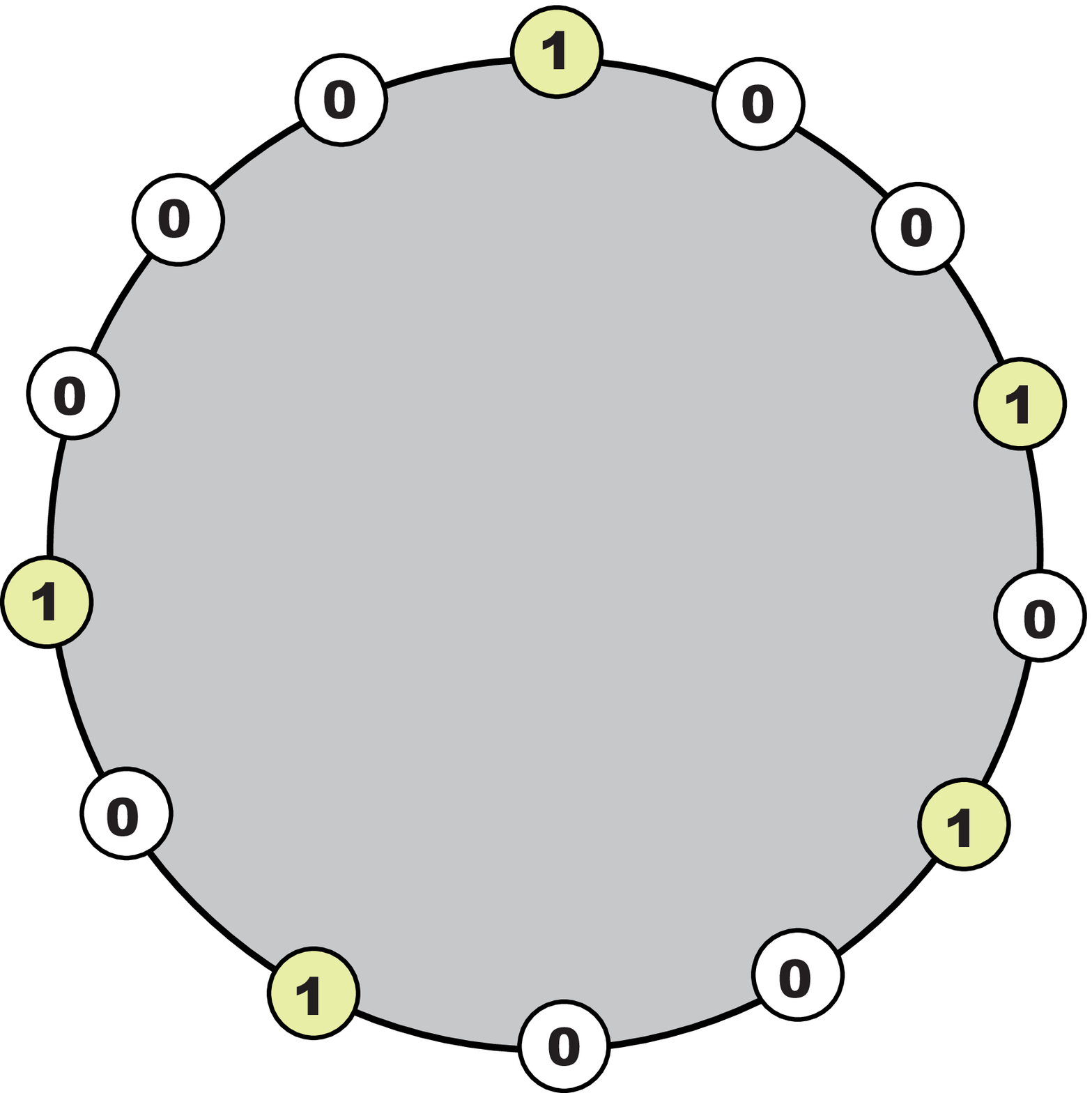}
\caption{\label{fig:SingleDetPAT}\textsc{Left:} Classical PAT sampling, where a single detector is moved around the object to collect individual  pressure
signal $p_j$. \textsc{Right:}  Compressed sensing approach   where each measurement  consists of a random  zero/one combination of individual pressure values.}
\end{figure}

\subsection{Compressed sensing PAT}
\label{sec:setup}

We define the unknown full sample pressure $p_j   \colon [0,2\rho]  \to \R$,  for $j \in \{1, \dots, N\}$  by
\eqref{eq:data}, \eqref{eq:sampling}, where $p$ is the solution of the two dimensional wave equation \eqref{eq:wave}.
We suppose that the (two dimensional) initial pressure distribution $f$ is smooth and compactly supported in $B_R(0)$ which implies that any $p_j$ is smooth and vanishes in a neighborhood of zero.
Furthermore we assume that the data \eqref{eq:data}  are sampled  finely enough to allow for $f$ to be reconstructed from $p_j$
by standard PAT reconstruction algorithms such as time
reversal \cite{BurMatHalPal07,HriKucNgu08,Treeby10} or filtered backprojection \cite{BurBauGruHalPal07,finch2007inversion,FinPatRak04,Hal14,Kun07a,XuWan05}.

Instead of measuring each pressure signal $p_j$ separately, we take $m$ compressed sensing measurements.  For each measurement, we select sensor locations $z_j$ with $j \in J_i$
at random and take the sum of the corresponding  pressure values $p_j$.
Thus the $i$-th measurement is given by
\begin{equation} \label{eq:cs2}
	y_i (t)
	:=
	\sum_{j\in J_i} 	
	p_j(t) \quad \text{ for } i  \in \set{  1, \dots, m}  \,,
\end{equation}
where $J_i \subset \set{1, \dots, N}$ corresponds to the set of all detector locations selected for the $i$-th measurement, and $m \ll N$ is the  number of compressed sensing measurements.

In practice, the compressed sensing measurements could be realized by summation of several detector channels, using a configurable matrix switch and a summing amplifier. Even more simply, a summing amplifier summing over all $N$ channels could be used, while individual detector channels $z_j$ are turned
{\tt on} or {\tt off}, using solid state relays. Thereby, only one ADC is required for one compressed sensing measurement. Performing $m$ compressed sensing measurements in parallel is facilitated by using $m$ ADCs in parallel, and the according number of matrix switches, relays, and summing amplifiers.

In the following we write
\begin{align}
\bp \colon \kl{0, \infty} &\to \R^N
\colon t \mapsto \bp(t) =
\kl{p_1(t), \dots, p_N(t)}^{\mathsf{T}} \,,
\\
\by \colon \kl{0, \infty} &\to \R^m
\colon t \mapsto \by(t) =
\kl{y_1(t), \dots, y_m(t)}^{\mathsf{T}} \,,
\end{align}
for the  vector of unknown complete pressure signals and the vector of the corresponding  measurement data, respectively. Further, we denote by
\begin{equation*}
	\Ao \in \set{0,1}^{m\times N}
	\text{ with entries }
	\Ao_{ij}:=
	\begin{cases}
	1,
	&\text{ for }  j \in J_i
	\\ 0,
	&\text{ for } j \not\in J_i \,,
	\end{cases}
\end{equation*}
the matrix whose entries in the $i$-th row correspond the sensor locations selected  for the  $i$-th measurement. In order to apply the exact recovery guarantees from
Subsection~\ref{sec:lossless}, we require that each row of $\Ao$ contains exactly $d$ ones, where $d \in \N$ is some fixed  number.  Practically, this  means that each detector location contributes to exactly $d$ of the $m$ measurements, which also guarantees that  the measurements are well calibrated and there is no bias towards some of the detector locations.

Recovering the complete pressure data \eqref{eq:data} from the compressed sensing measurements \eqref{eq:cs1} can be written as an uncoupled system of under-determined linear equations,
\begin{equation}\label{eq:system1}
	 \Ao \bp(t)   = \by(t)
	 \quad
	  \text{ for any } t \in [0,2\rho] \,,
\end{equation}
where $\Ao \bp(t) := \Ao \kl{\bp(t)}$ for any $t$.
From  \eqref{eq:system1}, we would like to recover  the complete set of pressure values $\bp(t)$ for all times $t \in [0,2\rho]$.
Compressed sensing results predict that under certain assumptions on the matrix $\Ao$ any $s$-sparse vector $\bp(t)$ can be recovered from $\Ao \bp(t)$ by means of sparse recovery algorithms like $\ell^1$-minimization.

Similar to many other applications (cf.~Section \ref{sec:background} above), however, we cannot expect sparsity in the original domain. Instead, one has $\bp(t) = \Psio \bx(t)$, where
$\Psio$ is an appropriate orthogonal transform  and  $\bx(t) \in \R^N$ is  a sparse coefficient vector. This yields a sparse recovery problem for $\bx(t)$ involving the matrix $\Ao \Psio$, which does not inherit the the recovery guarantees of  Subsection \ref{sec:lossless}. Hence we have to find a means to establish sparsity without considering different bases. Our approach will consist of applying a  transformation in the time domain that sparsifies the pressure in the spatial domain.
A further advantage of working in the original domain
is that the structure of $\Ao$ allows the use of specific efficient algorithms  like sparse matching pursuit~\cite{berinde2008practical} or certain sublinear-time algorithms
like \cite{jafarpour2009efficient,xu2007efficient}.

\subsection{Reconstruction strategy}
\label{sec:recon}

We denote by  $\mathcal{G}([0,2\rho])$ the
set of all infinitely differentiable functions
$g \colon  [0,2\rho] \to \R$ that vanish in a neighborhood of zero. To obtain the required sparsity, we will work with
a sparsifying transformation
\begin{equation} \label{eq:trafo}
	\To \colon  \mathcal{G}([0,2\rho])
	\to   \mathcal{G}([0,2\rho]) \,,
\end{equation}
that is,  $\To \bp (t) \in \R^N$ can be sufficiently well approximated by a sparse vector for any $t\in  [0,2\rho]$ and certain classes
of practically relevant data $\bp (t)$.
Here we use the convention that $\To$ applied to a vector valued function $\bg = (g_1, \dots,  g_k)$ is understood to be applied in each component separately,
that is,
\begin{equation} \label{eq:time-trafo}
	\To \bg (t) :=
	\kl{ (\To g_1)(t), \dots, (\To g_k)(t) } \,,
	\quad \text{ for } t \in [0, 2\rho] \,.
\end{equation}
We further require that $\To$ is an injective mapping, such that any $g \in \mathcal{G}([0,2\rho])$
can be uniquely recovered from the transformed data $\To g$.  See Subsection~\ref{sec:sparsifying} for the design of
such a sparsifying transformation.

Since any temporal transformation interchanges with $\Ao$, application of   the sparsifying temporal transformation $\To$
to the original system \eqref{eq:system1} yields
\begin{equation}\label{eq:system2}
	 \Ao \kl{ \To \bp(t) }  = \To \by(t)
	 \quad
	  \text{ for any } t \in [0,2\rho] \,.
\end{equation}
According to the choice of the temporal transform, the transformed pressure $\To \bp(t)$ can be well approximated by a vector that is sparse in the spatial  component.
We therefore solve, for any $t \in [0, T]$,  the following $\ell^1$-minimization problem
\begin{equation}\label{eq:opt}
\begin{aligned}
&\minimize_{\bq \in \R^N} &&\norm{\bq}_1\\
&\suchthat  &&\norm{\Ao\bq - \To\by(t)}_1
\leq \eta  \,,
\end{aligned}
\end{equation}
for some error threshold $\eta >0$.
As follows from Theorem~\ref{thm:stable}, the solution $\bq_\star(t)$ of  the $\ell^1$-minimization problem \eqref{eq:opt} provides an approximation to $\To \bp(t)$ and consequently we have $\To^{-1}  \bq_\star(t)  \simeq \bp(t)$ for all $t \in [0, 2\rho]$.
Note that  the use of $\ell^1$-norm
$\norm{\edot}_1$ for the constraint in \eqref{eq:opt} is required for the stable recovery guarantees for our particular choice of the  matrix $\Ao$ containing zero/one entries using data that is noisy and only approximately sparse.

As a further benefit,  compressed sensing measurements may show an increased signal to noise ratio. For that purpose consider Gaussian noise in the pressure data,  where instead of  the exact pressure data $\bp(t)$, we have noisy data   $\tilde\bp(t)  =  \bp(t) + \berr$.  For the sake of simplicity assume that the entries $\eta_j \sim\Ncal(0,\sigma^2)$ of   $\berr$  are independent and identically distributed.
The corresponding noisy (and rescaled) measurements
are then  given by
\[
\frac1{|J_i|}\sum_{j\in J_i}\kl{p_j(t)
+\eta_j }
=
\frac1{|J_i|}\sum_{j\in J_i}
p_j(t)
+
\frac1{|J_j|}\sum_{j\in J_i}\eta_j  \,.
\]
The variance in the compressed sensing measurements is therefore $\sigma^2/|J_i|$ compared to $\sigma^2$ in the  individual data $\tilde p_j(t)$. Assuming some coherent averaging
in the signal part this yields  an increased  signal
to noise ratio reflecting the inherent averaging of
compressed sensing measurements.

\subsection{Example of a sparsifying temporal transform}\label{sec:sparsifying}

As we have seen above, in order to obtain recovery guarantees for our  proposed compressed scheme,
we require a temporal transformation that sparsifies the pressure signals
in the angular component.
In this section, we construct an example  of such a sparsifying transform.

Since the solution of the two dimensional wave equation
can be reduced to the spherical means,  we will construct such a sparsifying
transform for the  spherical means
\[
 	\Mo f(z,r)
	:= \frac1{2\pi} \int_{\bbS^1}
	f(z + r \omega) \dx \sigma(\omega)
	\quad \text{ for } ( z, r)
	\in \partial B_R(0) \times  (0, \infty) \,.
\]
In fact, the solution of the two dimensional wave equation \eqref{eq:wave}
can be expressed in terms of  the spherical means via
$p(z,t) =  \partial_t \int_0^t r \Mo f (z,r)/ \sqrt{t^2-r^2} \rmd r $, see \cite{Joh82}.
By using standard tools for solving Abel type equations, the last expression can be explicitly inverted, resulting in $
(\Mo f)(z,r) =  2/\pi \int_0^r p (z,t)/\sqrt{r^2-t^2} \rmd t$
(see \cite{BurBauGruHalPal07,GorVas91}).
Hence any sparsifying transformation for the spherical means $\Mo f$
also yields a sparsifying transformation for the solution of the wave equation
\eqref{eq:wave}, and vice versa.

We found empirically, that $\partial_r r \Ho_r \partial_r \Mo f $ is sparse in the spatial direction
for any function $f$ that is the superposition of few indicator functions
of regular domains.  Here $\partial_r g$ denotes the
derivative in the radial variable,
\[
    \Ho_r g (z,r) =
    \frac1\pi \int_\R \frac{g (z,s)}{r-s}\dx s
    \,, \quad \text{ for } ( z, r)
	\in \partial B_R(0) \times (0, \infty) \,,
\]
the Hilbert transform of the function $g \colon \partial B_R(0) \times \R \to \R$ in the second component, and $r$  the multiplication operator that maps the function $(z,r) \mapsto g(z,r)$ to  the function $(z,r) \mapsto r g(z,r)$.  Further, the spherical means
$\Mo f \colon \partial B_R(0) \times \R \to \R$ are extended to an odd function
in the  second  variable.
For a simple radially symmetric phantom the sparsity of $\partial_r r \Ho_r \partial_r \Mo f $ is illustrated in  Figure~\ref{fig:SimplePhantomSparsity}.
Thus we can choose $\To = \partial_r r \Ho_r \partial_r$  as a sparsifying temporal transform for the spherical means.

\begin{figure}[tbh!]
\centering
\includegraphics[height =0.3\textwidth,width =0.3\textwidth]{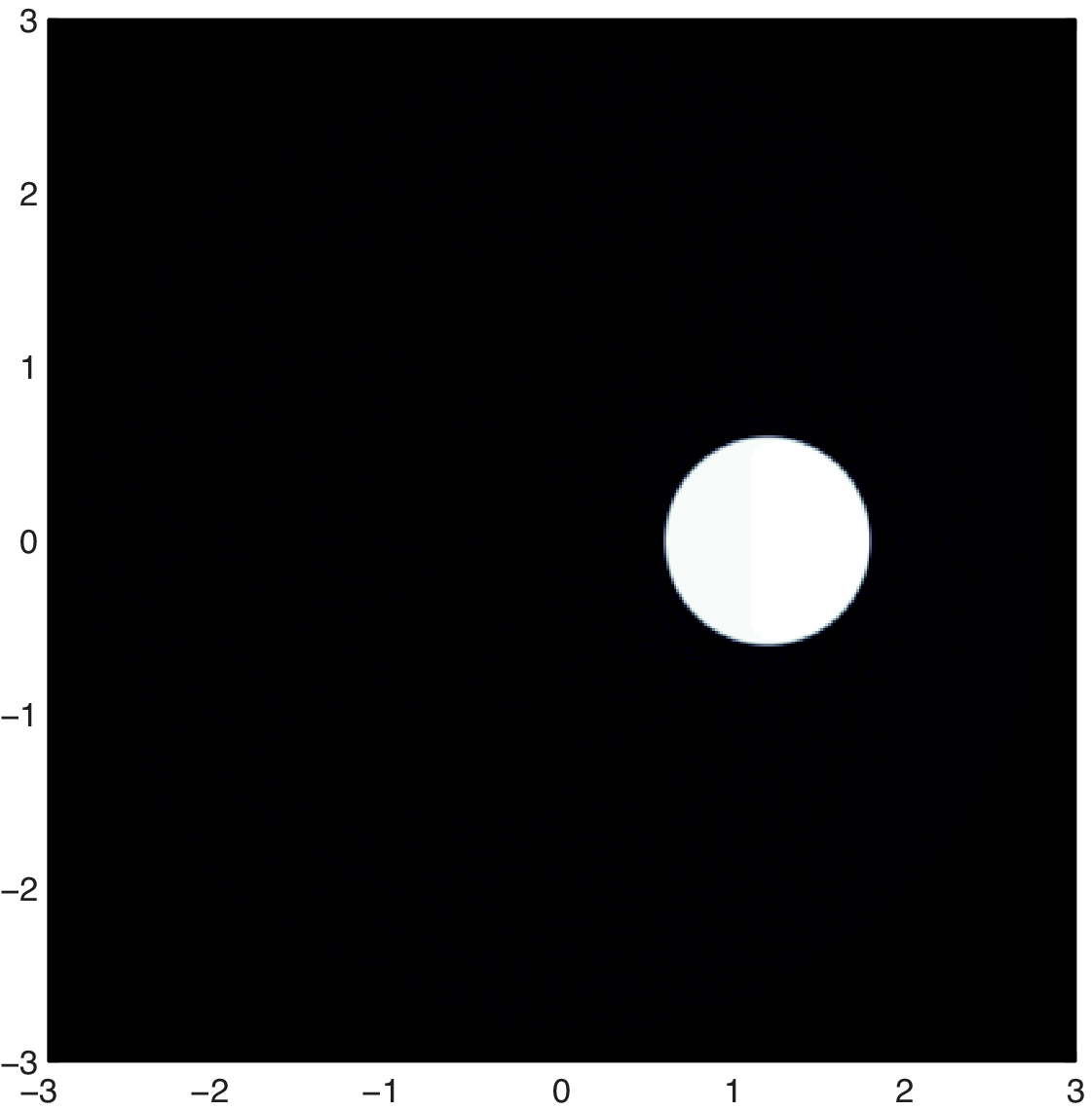}\quad
\includegraphics[height =0.3\textwidth,width =0.3\textwidth]{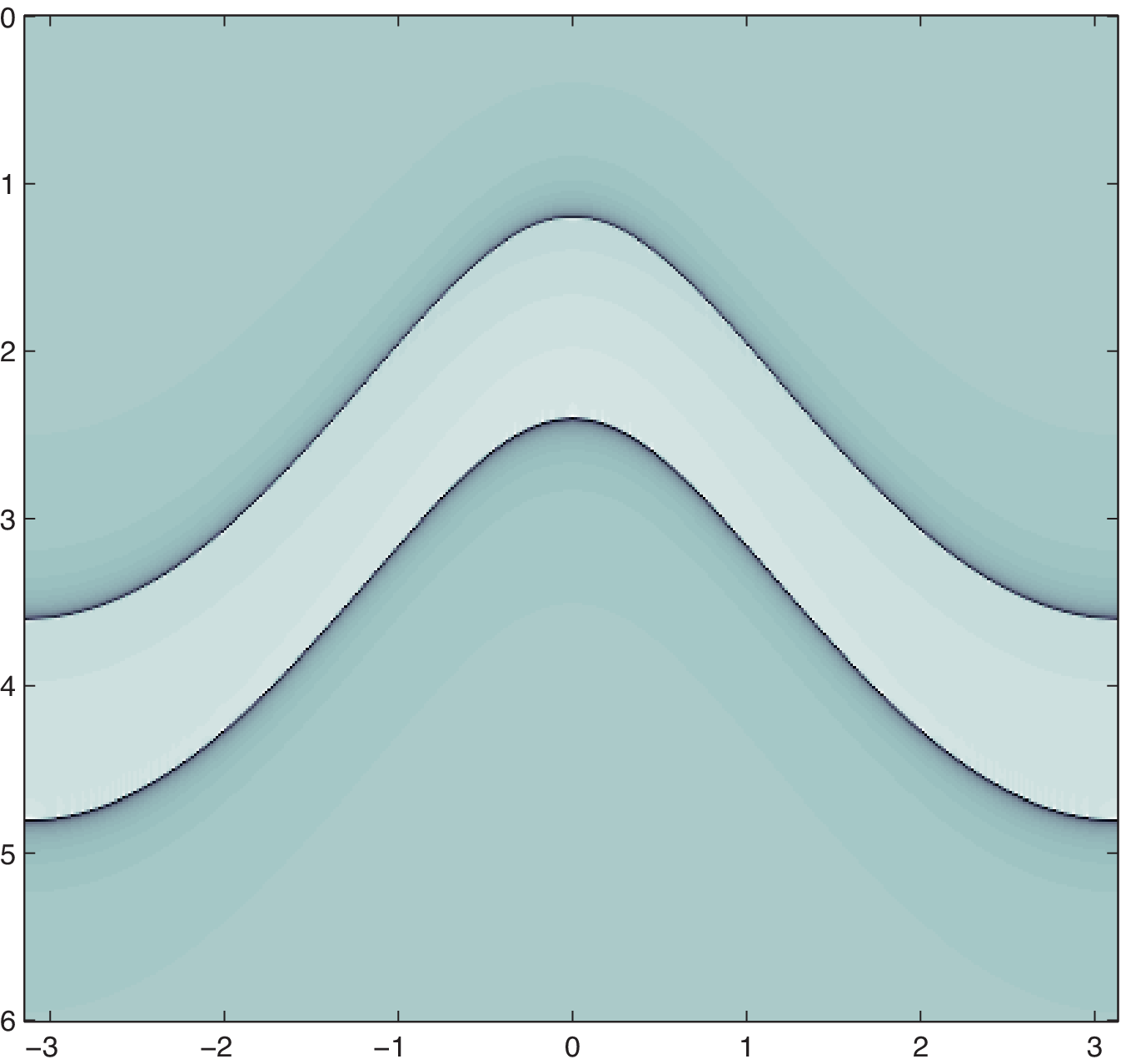}\quad
\includegraphics[height =0.3\textwidth,width =0.3\textwidth]{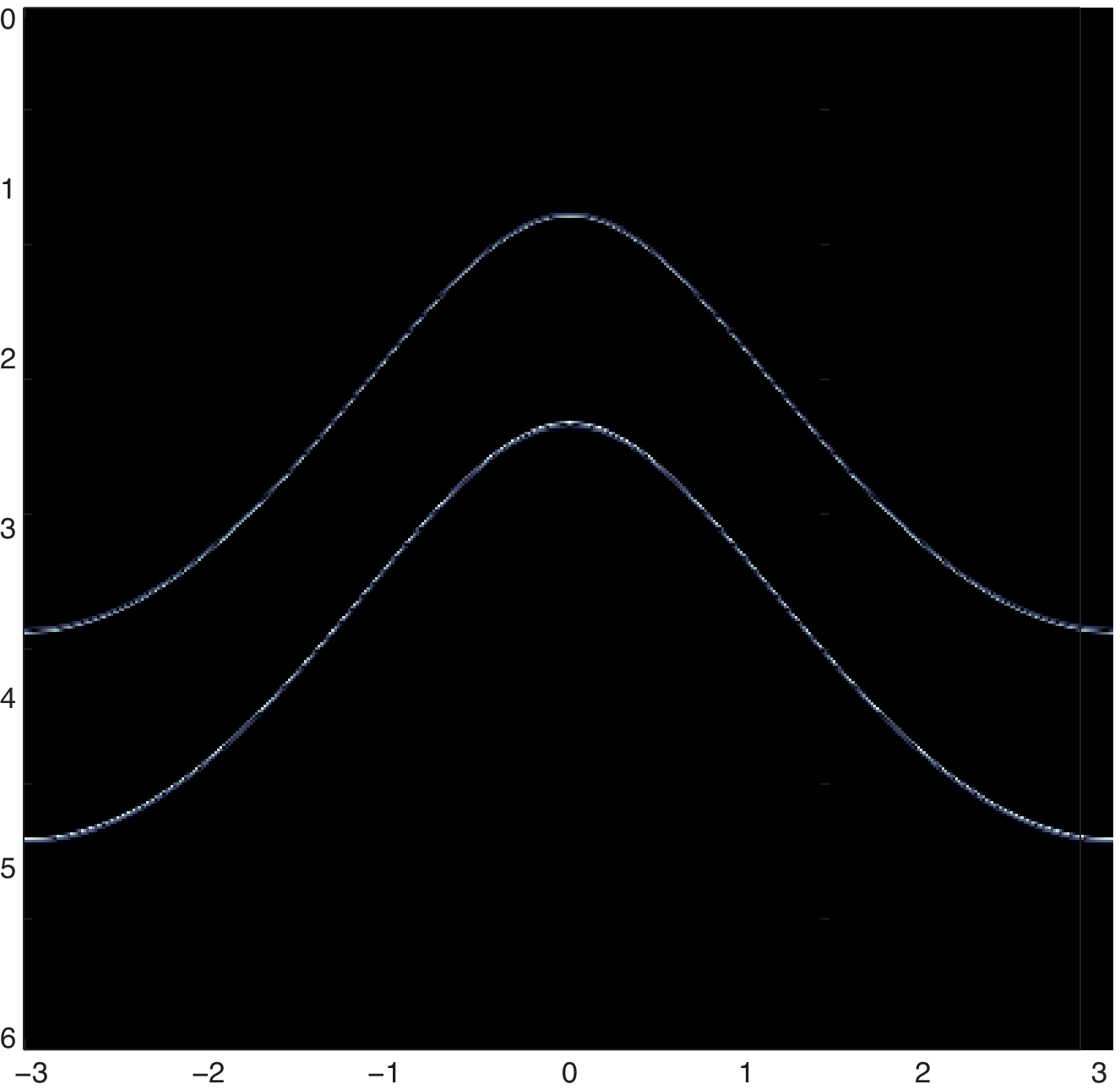}
\caption{\textsc{Sparsity induced by $\partial_r (r\Ho_r \partial_r)$.}
The left image shows the simple disc-like phantom $f$
(characteristic function of a disc), the middle image shows the
filtered spherical means  $r \Ho_r \partial_r \Mo f$ and the right
image shows the  sparse  data
$\partial_r (r\Ho_r \partial_r \Mo f)$.}
\label{fig:SimplePhantomSparsity}
\end{figure}

The function $r \Ho_r \partial_r \Mo f $ also appears in  the following formula
for recovering a function from its spherical means derived in
\cite{finch2007inversion}.

\begin{theorem}[Exact reconstruction formula for spherical means]\label{thm:fbp}
Suppose $f\in C^\infty(\R^2)$ is supported in the closure of  $B_R(0)$.
Then, for any $x\in B_R(0)$, we have
\begin{equation} \label{eq:fbp}
f(x) =
\frac1{2\pi R} \int_{\partial B_R(0)}
\kl{ r \Ho_r \partial_r \Mo  f}(z,|x-z|)
\,\dx s(z) \,.
\end{equation}
\end{theorem}

\begin{proof}
See \cite[Corollary 1.2]{finch2007inversion}.
\end{proof}

In the practical implementation the spherical means are given only for a discrete number of
centers $z_j \in \partial B_R(0)$  yielding semi-discrete data similar to \eqref{eq:data},
\eqref{eq:sampling}. Formula  \eqref{eq:fbp}  can easily be adapted to discrete or semi-discrete data yielding a filtered backprojection type reconstruction algorithm; compare \cite[Section~4]{finch2007inversion}.
So if we can find the filtered spherical means
\begin{equation}\label{eq:sm-disc}
\kl{ r \Ho_r \partial_r \Mo  f}(z_j,\edot) \quad \text{ for all detector locations $z_j \in \partial B_R(0)$}
\end{equation}
we can obtain the desired reconstruction of $f$ by applying the backprojection operator
(the outer integration in \eqref{eq:fbp}) to $r \Ho_r \partial_r \Mo  f$.

\subsection{Summary of the reconstruction procedure}
\label{sec:summary}

In this section, we combine and summarize
the  compressed sensing scheme for photoacoustic tomography as described in the previous sections.
Our proposed compressed sensing  and sparse recovery  strategy takes the following form.

\begin{enumerate}[topsep=0.5em,itemsep=0.5em,
label = (\textsc{cs}\arabic*)]
\item\label{cs1}
Create a matrix $\Ao \in\{0,1\}^{m\times N}$ as the adjacency matrix of a randomly selected
left $d$-regular bipartite graph.
That is, $\Ao$ is a random matrix consisting of zeros and ones only,
with exactly $d$ ones in each column.

\item\label{cs2}
Perform $m$ measurements,  whereby in the $i$-th measurement pressure signals corresponding  to the nonzero entries  in $i$-th row of $\Ao$ are summed up, see
Equations~\eqref{eq:cs1} and~\eqref{eq:system1}.
This results in measurement data $\Ao \bp(t) = \by(t)\in\R^m$ for any $t\in [0,2\rho]$.

\item\label{cs3}
Choose a  transform $\To$ acting in the temporal direction,
which sparsifies the pressure data $\bp$ along  the spatial direction; compare Equation~\eqref{eq:time-trafo}.

\item\label{cs4}
For any $t \in [0,2\rho]$ and some given threshold $\eta$, perform $\ell^1$-minimization
\eqref{eq:opt} resulting in a sparse vector $ \bq_\star(t)$ satisfying $\Vert \Ao \bq_\star (t) -  \To \by(t)\Vert_1\leq\eta$.

\item\label{cs5}
Use $\bp_\star(t) = \To^{-1}\bq_\star(t)$ as the input for  a standard PAT inversion algorithm for  complete data,
such as time reversal or filtered backprojection.
\end{enumerate}

As we have seen in Subsection~\ref{sec:lossless} the procedure
\ref{cs1}--\ref{cs5} yields a close approximation to the original function $f$ if  the transformed data $\To \bp(t)$ are
sufficiently sparse in the spatial direction.
The required sparsity level is hereby given by the expander-properties of the matrix $\Ao$. Note that for exact data and exactly sparse data, we can use the error threshold $\eta = 0$.
In the more realistic scenario of noisy data and $\To \bp(t)$ being
only approximately sparse, we solve the optimization problem~\eqref{eq:opt} to yield a near optimal solution with error level bounded  by the noise level.

\section{Numerical results}\label{sec:numerical}
\label{sec:num}
To support the theoretical examinations in the previous sections,
in this section we present some simulations using
the proposed compressed sensing method. We first present reconstruction
results  using simulated data and then show reconstruction results using
experimental data.

\subsection{Results for simulated data}
\label{sec:sim}

As in Subsection~\ref{sec:sparsifying}, we work with the equivalent notion of the spherical means instead of  directly working with the solution of the wave equation \eqref{eq:wave}.
In this case the compressed sensing measurements provide data
\begin{equation} \label{eq:cs-m}
	y_i(r)  =
	\sum_{j \in J_i} m_j \kl{r}
	\quad \text{ for } i \in \set{ 1, \dots,  m} \text{ and } t \in  [0,2\rho] \,,
\end{equation}
where  $m_j   = (\Mo f) \kl{z_j,  \edot}$ denote the spherical
means collected at the $j$-the detector location $z_j$. We further denote by
$\boldm(t) =  \kl{m_1(t), \dots, m_N(t)}^{\mathsf{T}}$ the vectors of  unknown complete
spherical means and by $ \by(t) = \kl{y_1(t), \dots, y_m(t)}^{\mathsf{T}}$ the
vector of compressed sensing measurement data. Finally, we  denote by $\Ao \in \set{0,1}^{m\times N}$
the compressed sensing matrix such that  \eqref{eq:cs-m} can be rewritten in the form $\Ao \boldm = \by$.

As proposed  in Subsection \ref{sec:sparsifying} we use
$\To  = \partial_r (r\Ho_r\partial_r)$ as a sparsifying transform for the spherical means.
An approximation to $\partial_r r \Ho_r \partial_r \Mo  f$ can be obtained from compressed sensing measurements in combination via $\ell^1$-minimization.
For the recovery of the original function from the completed measurements, we use one of the inversion formulas of \cite{finch2007inversion}
presented given in Theorem \ref{thm:fbp}. Recall that this inversion formulas  can be implemented  by applying the circular back-projection to the filtered spherical means
$r \Ho_r\partial_r \Mo f$.

In order to obtain an approximation to the data \eqref{eq:sm-disc} from the sparse intermediate
reconstruction   $\partial_r (r \Ho_r \partial_r \Mo  f)(\edot , r)$, one has to perform one numerical integration along  the second
dimension after  the $\ell^1$-minimization process. We found that this numerical integration introduces artifacts in the reconstruction of $r \Ho_r \partial_r \Mo  f$, required for application of the  inversion formula~\eqref{eq:fbp},  see the middle image
in Figure \ref{fig:discCS}.
These artifacts also  yield some undesired blurring in the final reconstruction of $f$; see the middle image in Figure~\ref{fig:discCS2}.

\begin{figure}[tbh!]
\centering
\includegraphics[width=0.3\textwidth, height=0.3\textwidth]{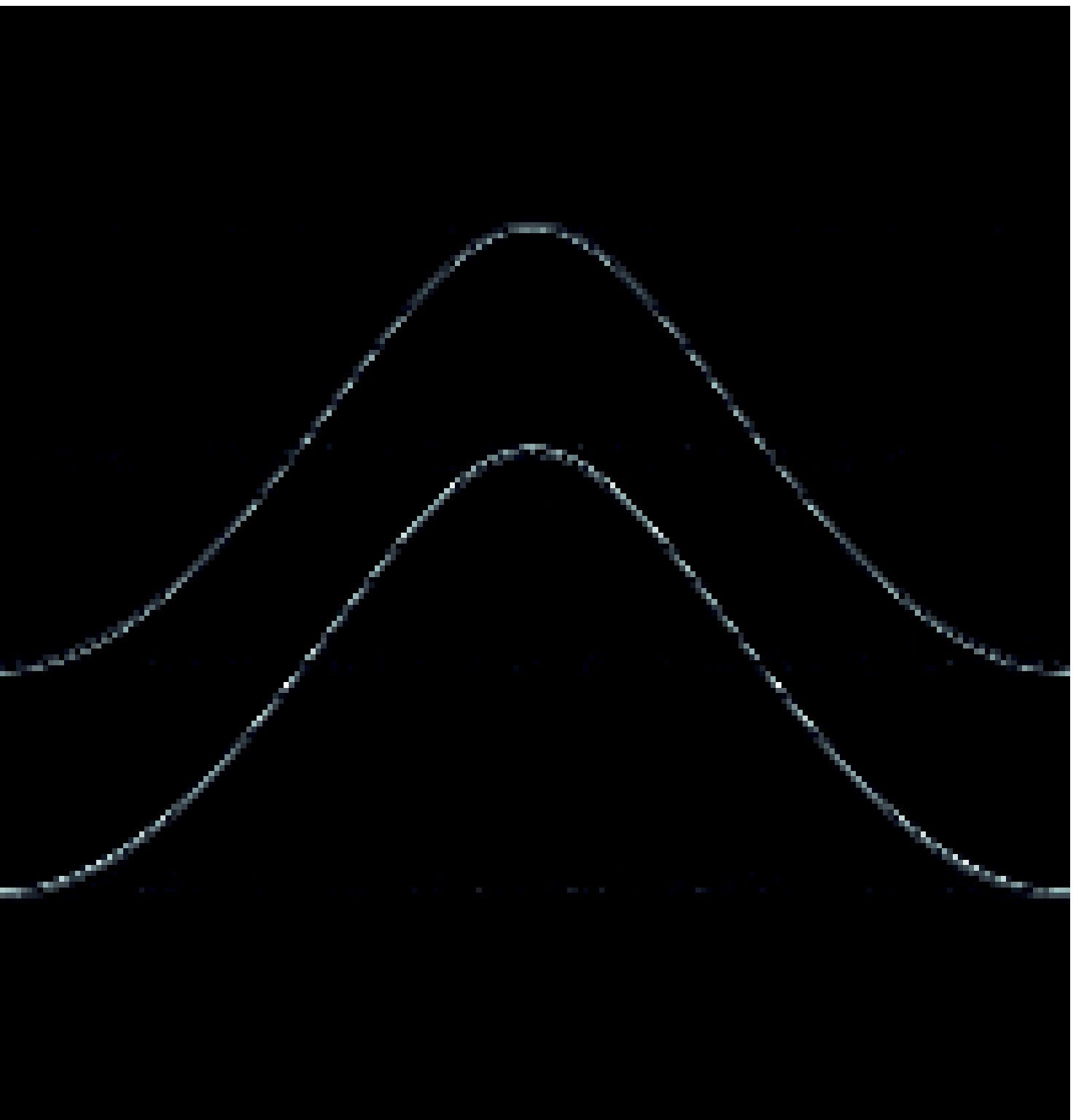}
\quad
\includegraphics[width=0.3\textwidth, height=0.3\textwidth]{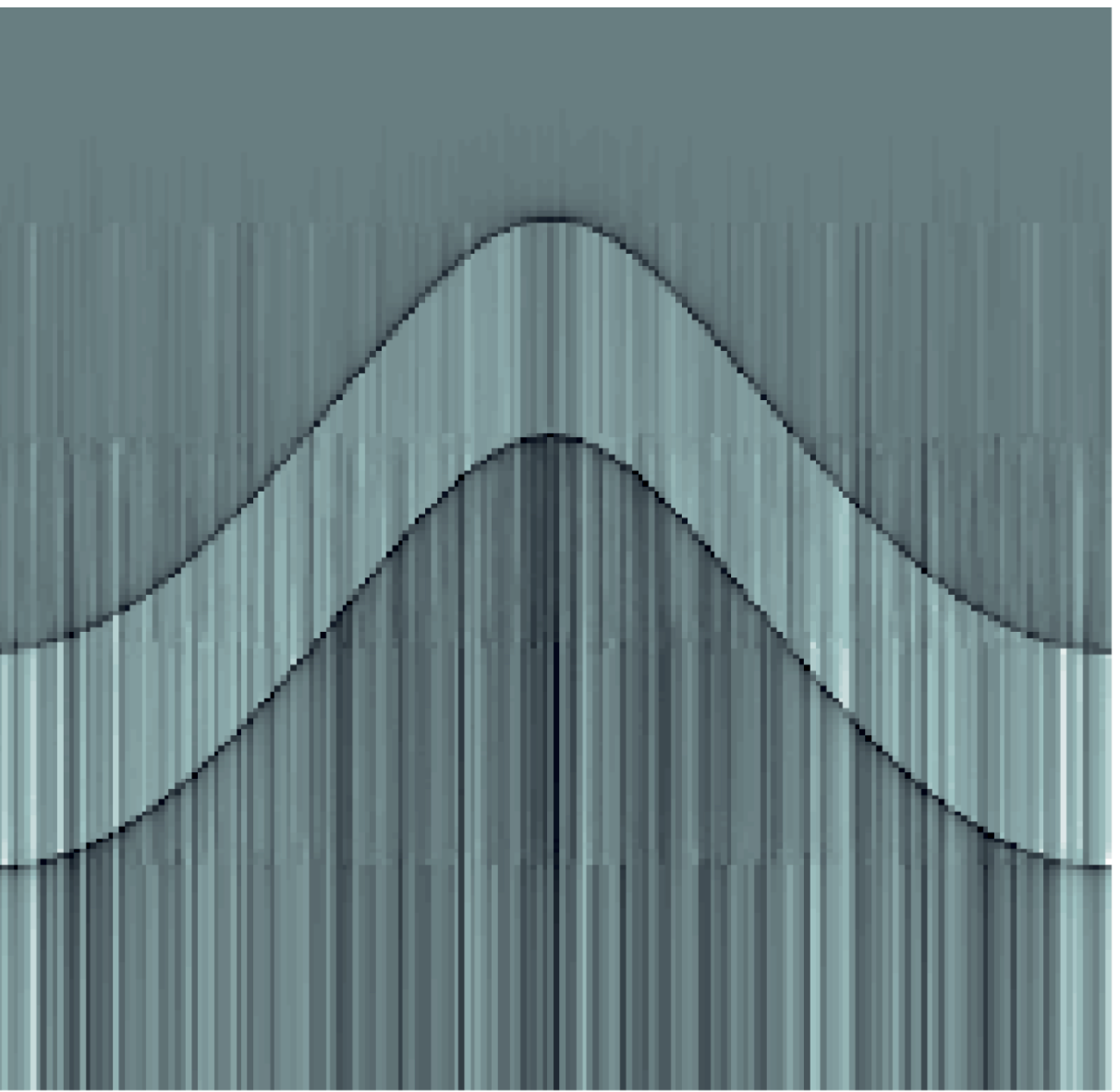}
\quad
\includegraphics[width=0.3\textwidth, height=0.3\textwidth]{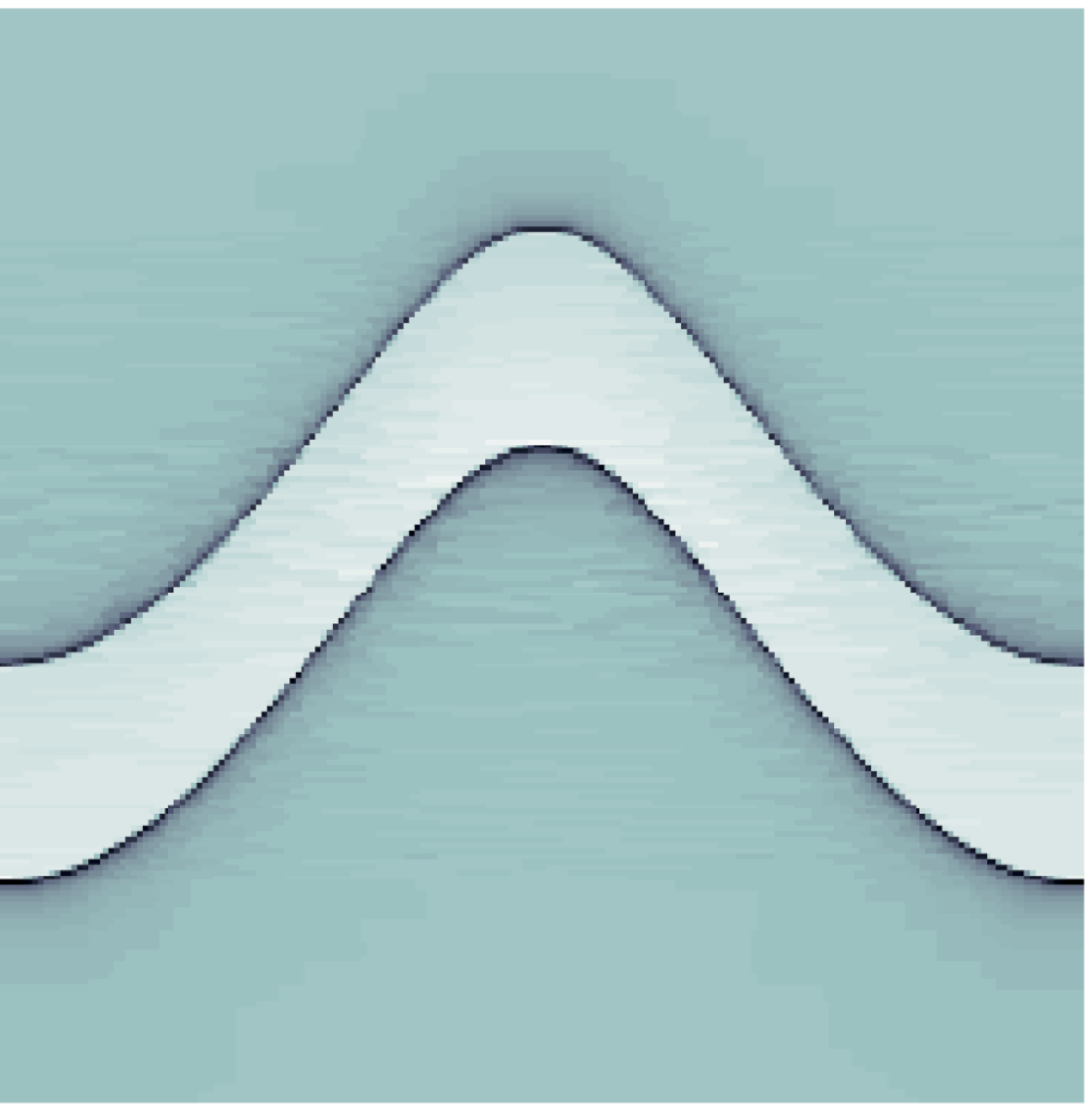}
\caption{\textsc{Reconstruction of filtered spherical means for $N=200$ and $m=100$.}
Left: Reconstruction of $\partial_r r \Ho_r\partial_r  \Mo f$  using $\ell^1$-minimization. Center: Result of integrating the $\ell^1$-reconstruction in the radial direction.
Right:  Result of directly recovering $r \Ho_r\partial_r  \Mo f$ by TV-minimization.}
\label{fig:discCS}
\end{figure}

In order to overcome the artifact introduced by the numerical integration, instead of applying an additional radial derivative to
$r\Ho_r\partial_r$ to  obtain sparsity in the spatial direction,  we will apply one-dimensional total variation minimization (TV)
for directly recovering $(r \Ho_r \partial_r \Mo  f)(\edot , r)$.
Thereby we avoid performing numerical integration on the reconstructed sparse data. Furthermore,  this yields much better results in terms of image quality of the final image $f$.

In our interpretation, this performance discrepancy is comparable to the difference between uniform and variable density samples for subsampled Fourier measurements. While \cite{candes2006robust} proves recovery of the discrete gradient, this does not carry over to the signal in a stable way -- a refined analysis was required \cite{needell2013stable, KW13}. Similarly, we expect that a refined analysis to be provided in subsequent work can help explain the quality gap between $\ell^1$ and TV minimization that we observe in our scenario.

In order to approximately recover
$r \Ho_r\partial_r \Mo f$ from the compressed
sensing measurements, we  perform,  for any $r \in [0,2\rho]$,
one-dimensional discrete TV-minimization
\begin{equation} \label{eq:tv-min}
	\norm{ \Ao \bq -  (r \Ho_r\partial_r \by) (r)}^2_2
	+   \la \norm{\bq}_{\rm TV}  \to \min_{\bq \in \R^N} \,.
\end{equation}
Here  $\norm{\bq}_{\rm TV}  = 2\pi/N \sum_{j=1}^{N} \abs{q_{j+1} - q_j}$ denotes the discrete total variation using the periodic extension $q_{N+1} := q_1$.
The one-dimensional total variation minimization  problem \eqref{eq:tv-min} can be  efficiently solved
using the fast iterative shrinkage thresholding algorithm (FISTA) of Beck and Teboulle~\cite{beck2009fast}.
The required proximal mapping for the total variation can be computed  in $\mathcal O(N)$ operation counts
by the  tautstring algorithm (see \cite{mammen1997locally,davies2001local,GraObe08}).
The approximate solution  of \eqref{eq:tv-min} therefore only requires  $\mathcal O(N m N_{\rm iter})$
floating point operations, with $N_{\rm iter}$ denoting the number of  iterations in the FISTA.
Assuming the radial variable to be discretized using $\mathcal O(N)$ samples,  the whole data completion procedure by \eqref{eq:tv-min} only requires $\mathcal O(N^2 m N_{\rm iter})$ operation counts.
Since we found that fewer than 100 iterations in the FISTA are often sufficient for accurate results,
the numerical effort of data completion is only a few times higher than that of standard reconstruction algorithms in PAT.

\begin{figure}[tbh!]
\centering
\includegraphics[width=0.3\textwidth, height=0.3\textwidth]{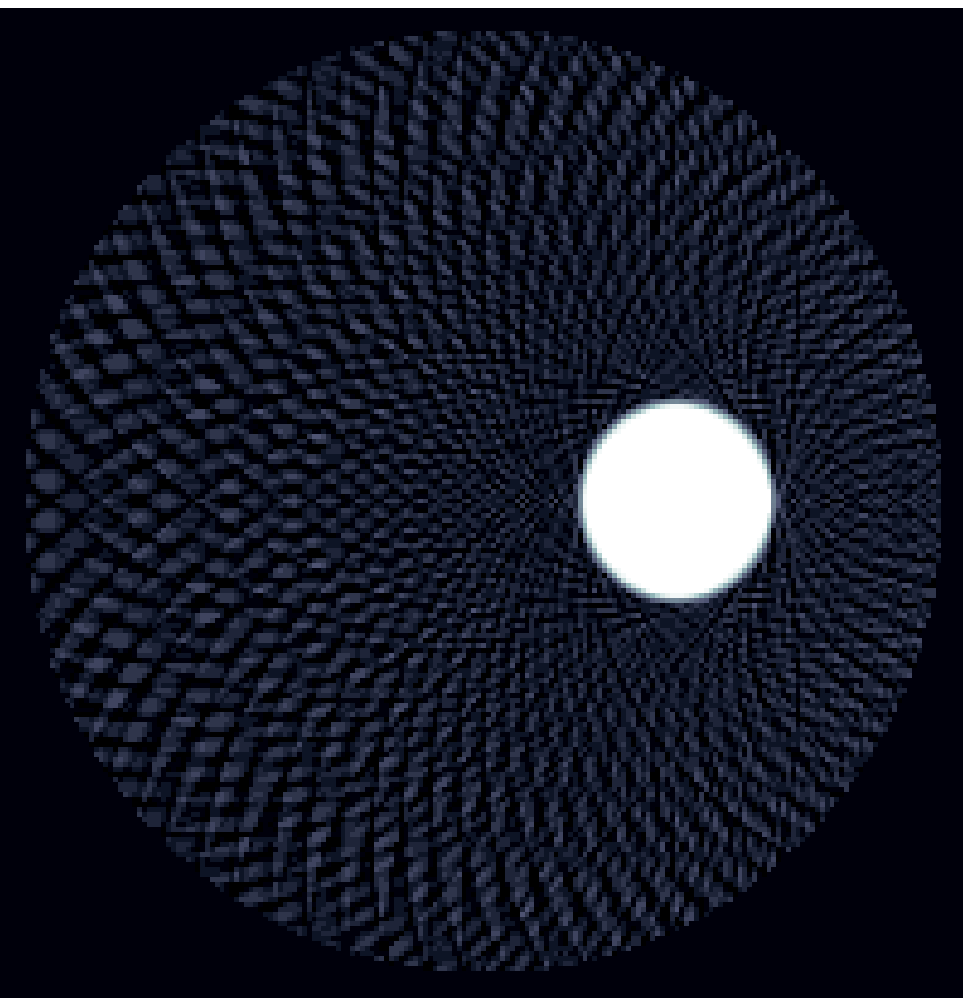}
\quad \includegraphics[width=0.3\textwidth, height=0.3\textwidth]{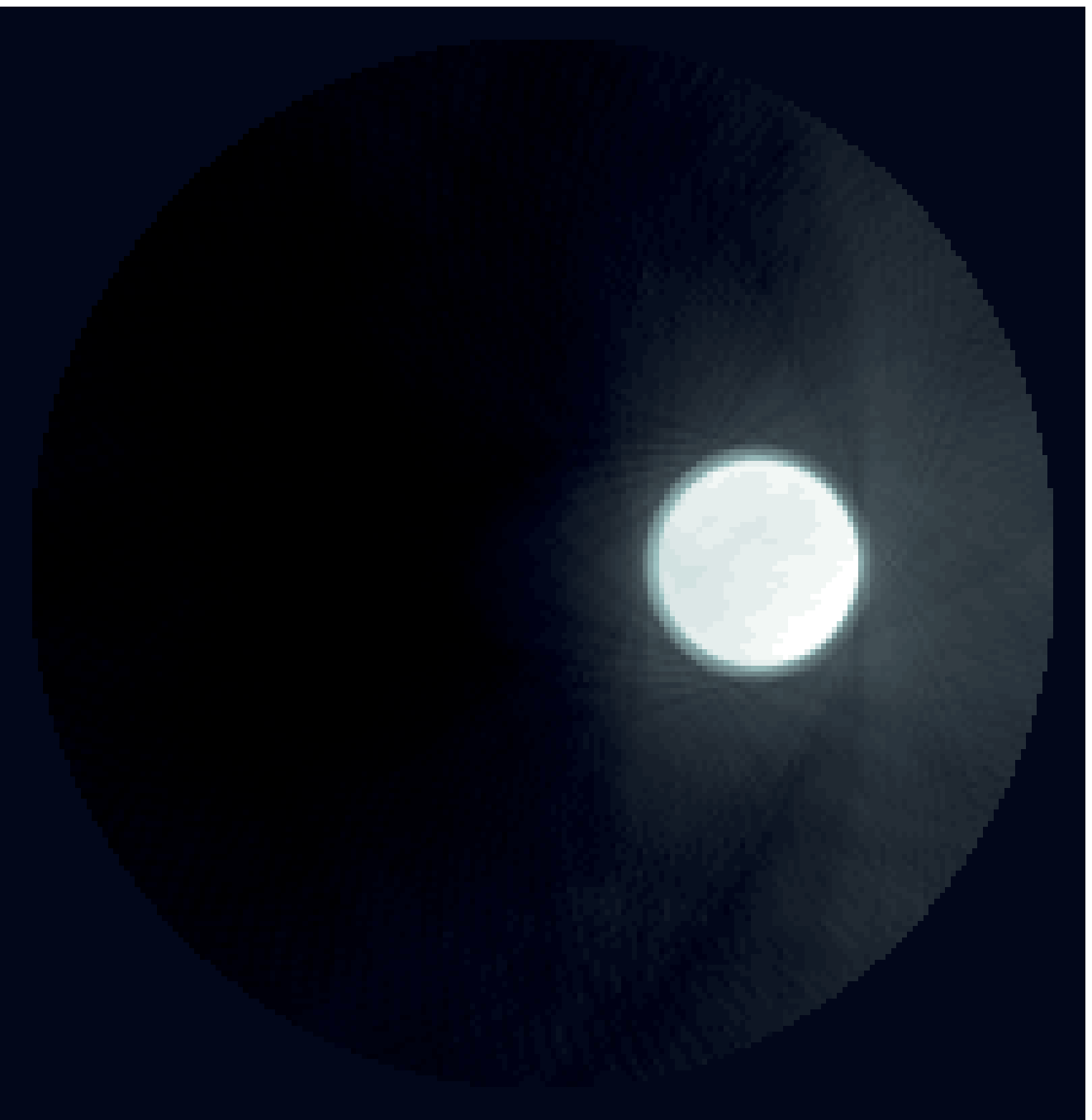}
\quad\includegraphics[width=0.3\textwidth, height=0.3\textwidth]{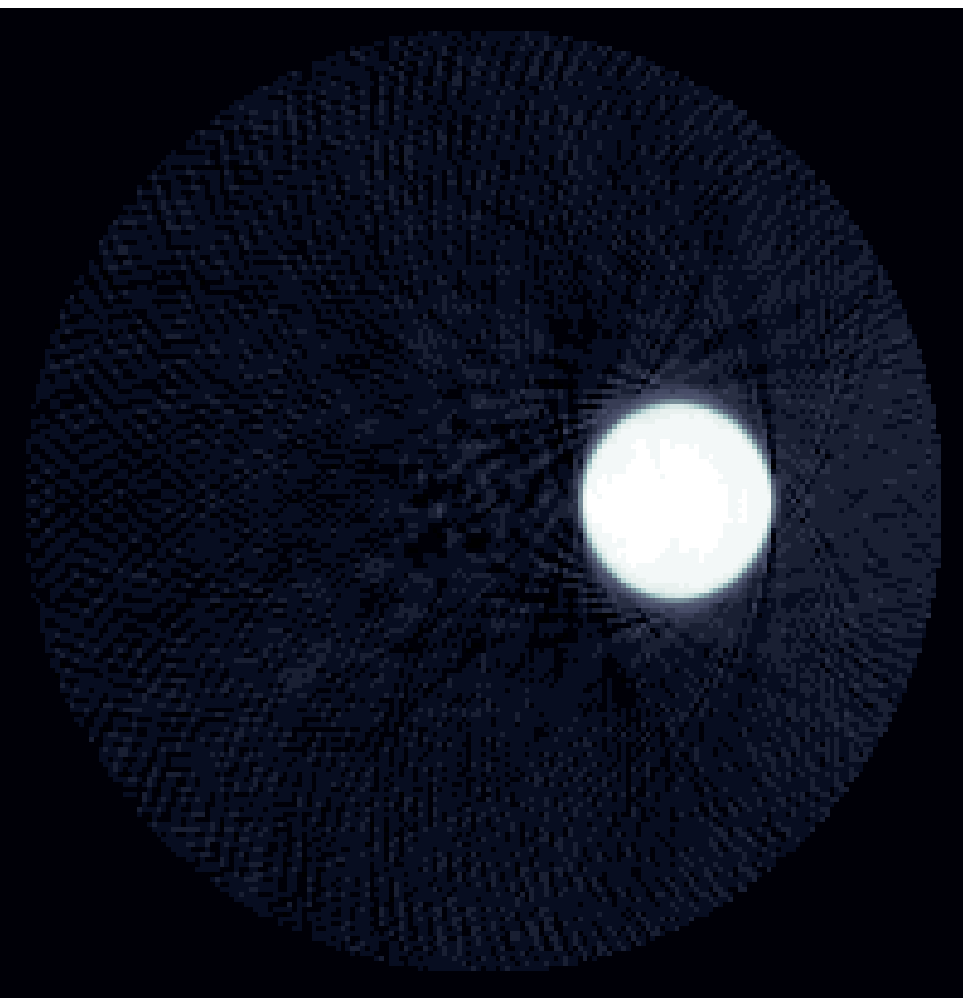}
\caption{\textsc{Reconstruction results for disc-like phantom using $N=200$ and $m=100$.}
Left: Reconstruction from 100 standard measurements. Center: Compressed sensing reconstruction using
 $\ell^1$-minimization. Right: Compressed sensing reconstruction using  TV-minimization. 
 The reconstruction from standard measurements contains under-sampling artifacts 
 which are not present in the compressed sensing reconstructions. Further, the use of 
 TV-minimization yields much less blurred results than the use of  $\ell^1$-minimization.}
\label{fig:discCS2}
\end{figure}

Figures \ref{fig:discCS} and \ref{fig:discCS2} show results of simulation studies for a simple phantom,  where the initial pressure distribution $f$ is the characteristic function of  a disc.
For the compressed sensing reconstruction, we used $m = 100$ random measurements instead of $N = 200$  standard point measurements. As one can see from the right image in Figure~\ref{fig:discCS2}, the completed data $r \Ho_r\partial_r  \Mo f$, for this simple phantom, is recovered almost perfectly from the compressed sensing measurements  by means of TV-minimization.  The reconstruction  results in Figure~\ref{fig:discCS2}
show that the combination of our compressed sensing approach with TV-minimization yields  much better results than  the use of
$\ell^1$-minimization.
Fur comparison purpose, the left image in Figure~\ref{fig:discCS2} shows the reconstruction from 100 standard measurements.
Observe that the use of $m = 100$ random measurements  yields better result than the use of the $100$ standard measurements, where artifacts due to spatial under-sampling are clearly visible.

\subsection{Results for real measurement data}

Experimental  data have been acquired by scanning a single integrating line detector on a circular path around a phantom.
A bristle with a diameter of $\unit[120]{\mu m}$ was formed to a knot and illuminated from two sides with pulses from a frequency doubled Nd:YAG laser with a wavelength of $\unit[532]{nm}$.  The radiant exposure for each side was below $\unitfrac[7.5]{mJ}{cm^{2}}$.
Generated photoacoustic signals have been detected by a graded index polymer optical fiber being part of a Mach-Zehnder interferometer, as described in \cite{GruJBO10}. Ultrasonic signals have been demodulated using a self-made balanced photodetector, the high-frequency output of which was sampled with a 12-bit data acquisition device. A detailed explanation of the used photo-detector and electronics can be found in \cite{BauSPIE13}. The polymer fiber detector has been scanned around the phantom on a circular path with a radius of $\unit[6]{mm}$
and photoacoustic signals have been acquired on 121 positions. The scanning curve was not closed, but had an opening angle of
$\unit[\pi/2]{rad}$. Hence photoacoustic signals have been acquired between $\unit[\pi/8]{rad}$ and $\unit[15\pi/8]{rad}$.

Using these experimental data,  compressed sensing data have been generated, where each detector location was used $d=10$ times and $m = 60$ measurements are made in total.
The reconstruction of the complete measurement data has been obtained by one-dimensional discrete TV-minimization \eqref{eq:tv-min} as suggested in the Section \ref{sec:sim}. The measured and the recovered complete  pressure data are shown in the top row in Figure \ref{fig:PatRec}.
The bottom row in Figure \ref{fig:PatRec} shows the reconstruction results from 121 standard measurements (bottom left)  and the reconstruction from 60 compressed sensing measurements (bottom center).
Observe that there is only a small difference between the reconstruction results. This clearly
demonstrates the potential of our compressed sensing scheme \ref{cs1}--\ref{cs5} for decreasing the number of measurements while keeping image quality. 
For comparison purpose we also display the reconstruction using  60 standard measurement  (bottom right). Compared  to the compressed sensing reconstruction using the same number of measurements the use of standard measurements shows  significantly more artifacts which are due to spatial
under-sampling.

\begin{figure}[tbh!]
\centering
\includegraphics[height=0.4\textwidth,width =0.4\textwidth]{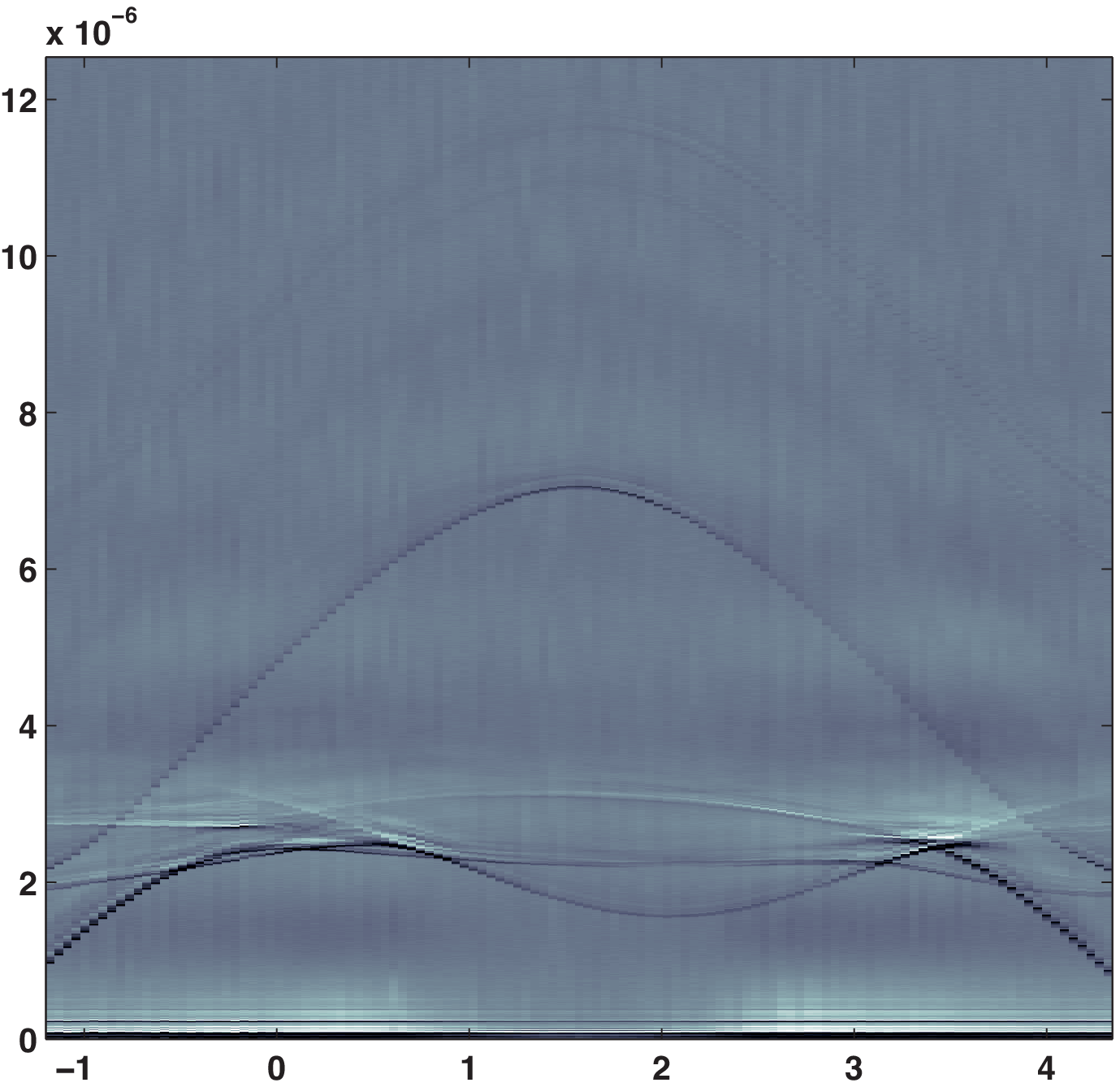}\qquad \includegraphics[height=0.4\textwidth,width =0.4\textwidth]{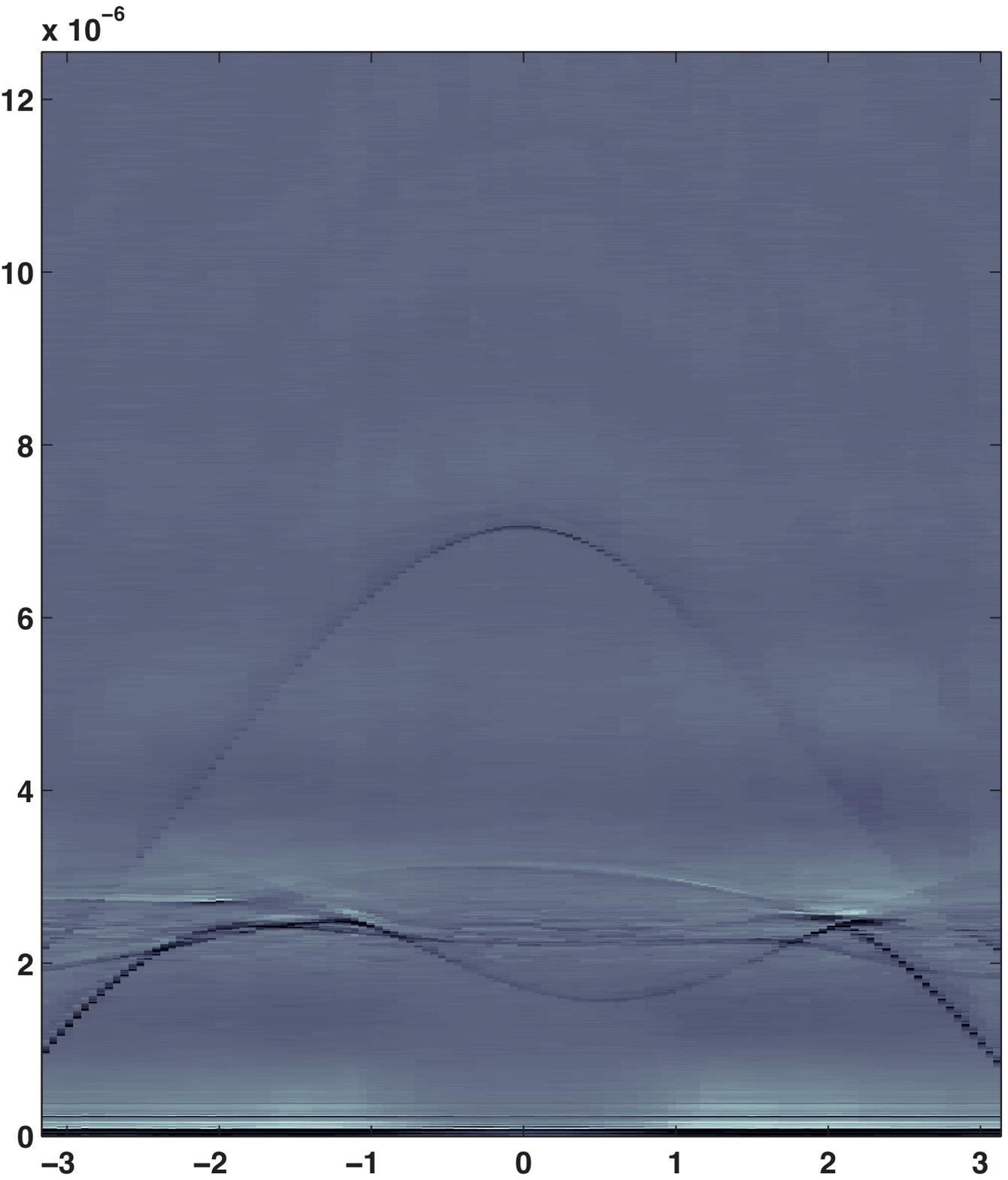}
\\[1em]
\includegraphics[width=0.2\textwidth,height =0.5\textwidth]{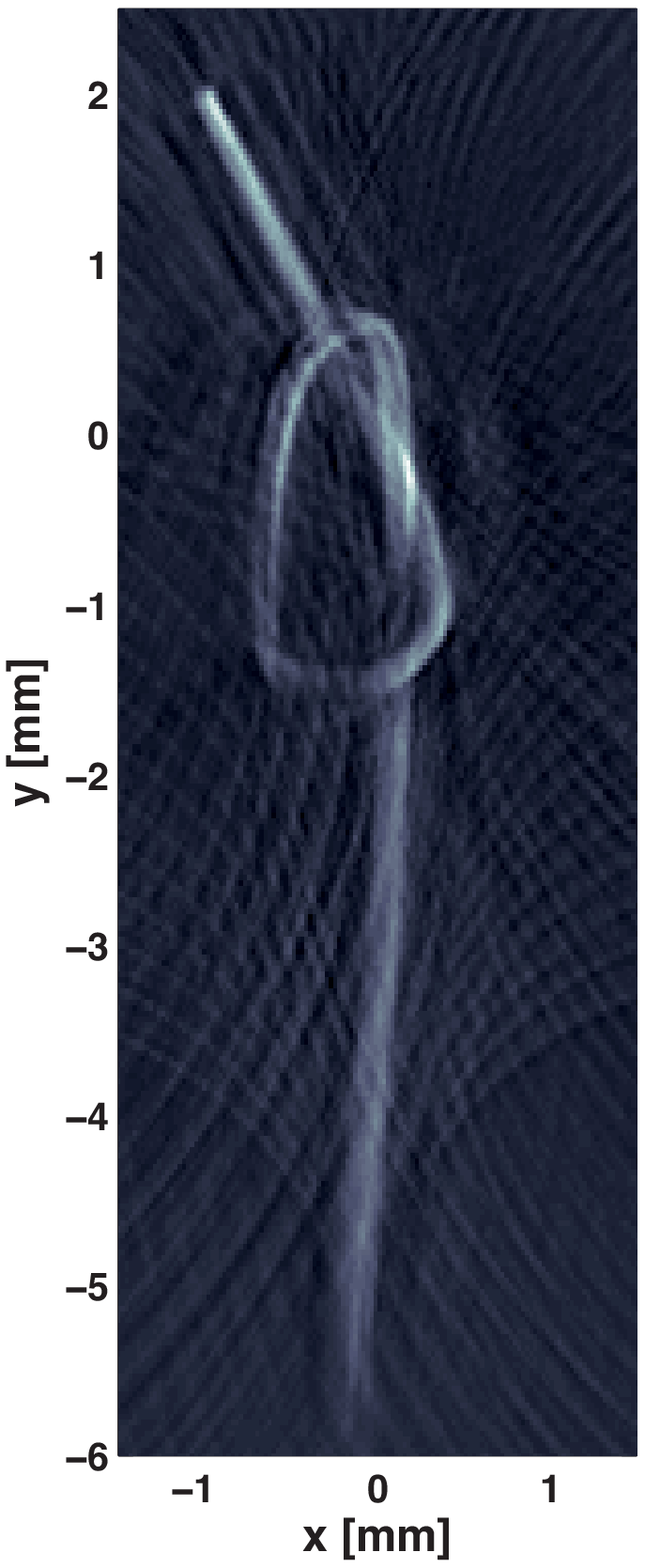}\qquad
\includegraphics[width=0.2\textwidth,height =0.5\textwidth]{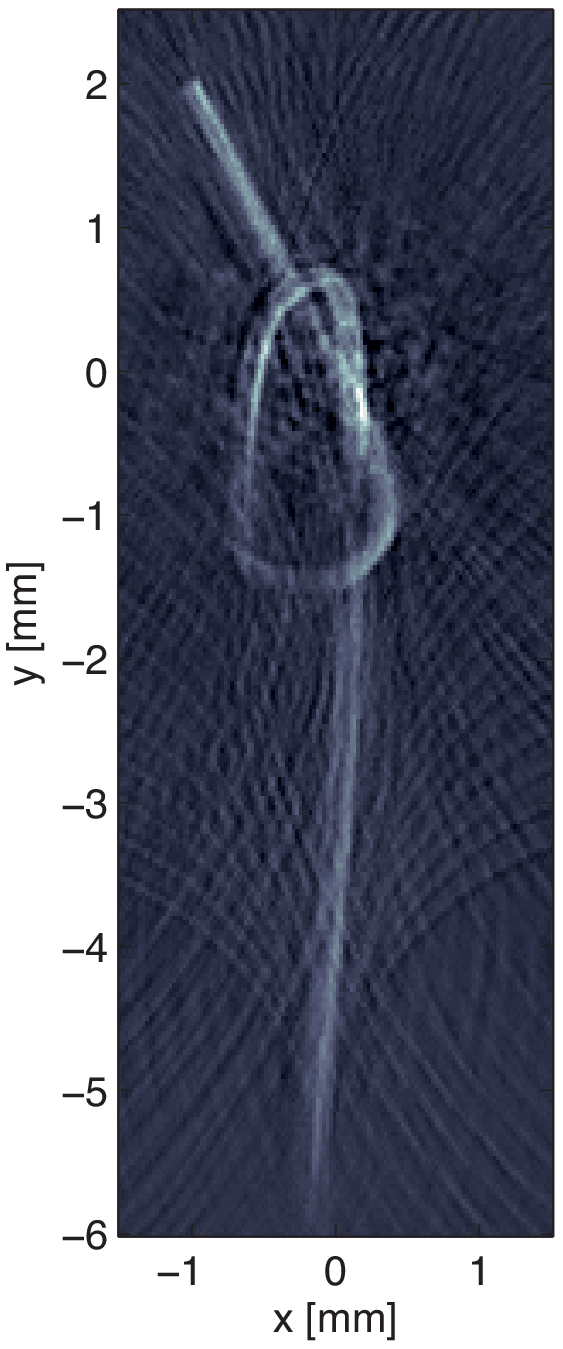}
\qquad
\includegraphics[width=0.2\textwidth,height =0.5\textwidth]{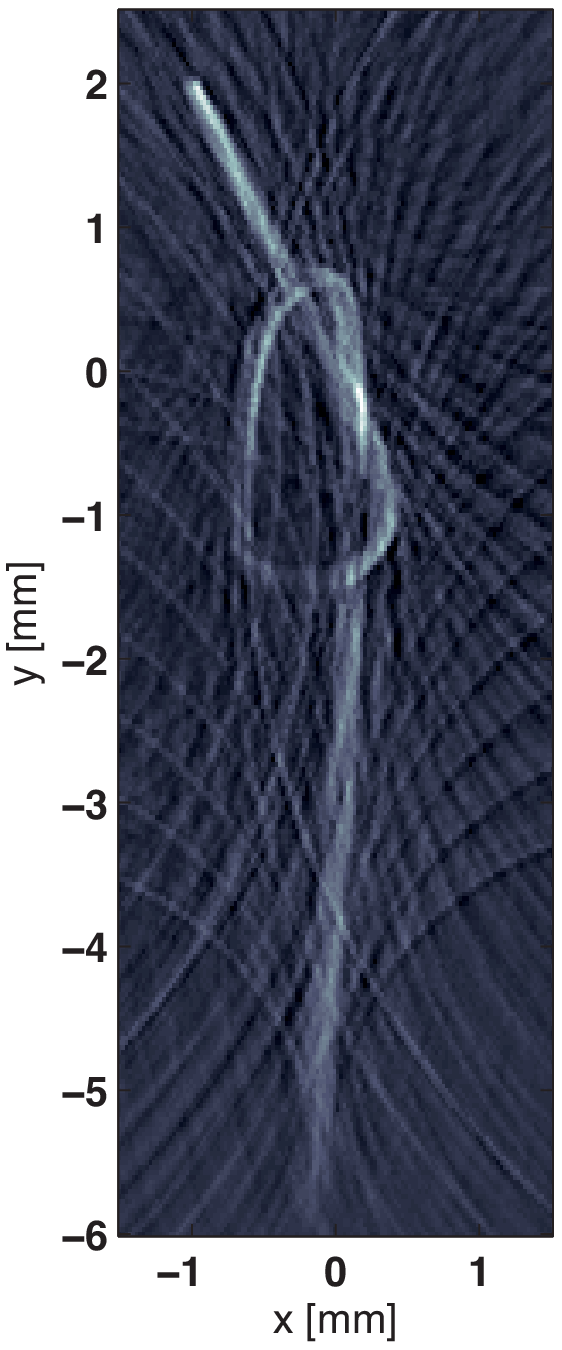}\caption{\textsc{Reconstruction results for PAT measurements.} Top left: Measured pressure data for $N=121$ detector positions. Top right: Compressed sensing reconstruction of the full pressure data from $m=60$ compressed measurements.
Bottom left: Reconstruction from the full measurement data using $N=121$ detector positions.
Bottom center: Reconstruction from $m = 60$ compressed sensing measurements. Bottom right: Reconstruction from half of the measurement data using $60$ detector positions. }
\label{fig:PatRec}
\end{figure}

\section{Discussion and outlook}
\label{sec:discussion}
In this paper, we proposed a novel approach to compressive sampling for photoacoustic tomography
using integrating line detectors providing recovery guarantees for suitable datasets.
Instead of measuring pressure data $p_j$ at any of the $N$ individual line detectors, our approach uses $m$ random combinations of $p_j$ with $m\ll N$ as measurement data. The reconstruction strategy  consists of first recovering the pressure values  $p_j$ for  $j \in \set{1, \dots, N}$  and then applying a standard PAT reconstruction for obtaining the final photoacoustic image.
For recovering the individual pressure data, we propose to apply a sparsifying transformation that acts in the temporal variable and makes the data sparse in the angular component.
After applying such a transform the complete pressure data is  recovered by solving a set of one dimensional $\ell^1$-minimization problems. This  decomposition also makes our  reconstruction algorithm numerically efficient.

Although we focused on PAT using integrating line detectors, we emphasize that a similar framework
can be developed for standard PAT based on the three dimensional wave equation and a two dimensional
array of point-like detectors. In such a situation, finding a sparsifying transform is even simpler. Recalling that $N$-shaped profile of the thermoacoustic pressure signal  induced by the
characteristic function of a ball suggests to use $\bp \mapsto \partial^2_t \bp$ as a sparsifying transform.

Note that the recovery guarantees in this paper crucially depend on the choice of an appropriate sparsifying transform.
In Subsection \ref{sec:sparsifying} we  proposed a candidate for
such a transformation that works well in our numerical
examples.  A more theoretical study of such sparsifying transforms and the resulting recovery guarantees
for simple (piecewise constant) phantoms is postponed
to further research. In this context, we will also investigate the use of different
sparsifying temporal transforms, such as the 1D wavelet
transform in the temporal direction. Further research includes using a fixed number of
detectors in each  measurement process. This requires novel
results for right $k$-regular expander graphs and
compressive sampling.

\section*{Acknowledgements}

The authors thank Johannes Bauer-Marschallinger and Karoline Felbermayer for performing measurements and for providing the experimental data. The work of M.~Sandbichler has been supported by a doctoral fellowship from the University of Innsbruck.
 The work of F.~Krahmer has been supported in parts by the German Science Foundation (DFG) in the context of the Emmy Noether Junior Research Group KR 4512/1-1 (RaSenQuaSI) and in parts by the German Federal Ministry of Education and Reseach (BMBF) through the cooperative research project ZeMat.
 T.~Berer and P.~Burgholzer have been supported by the Austrian Science Fund (FWF), project number S10503-N20, by the Christian Doppler Research Association, the Federal Ministry of Economy, Family and Youth, the European Regional Development Fund (EFRE) in the framework of the EU-program Regio 13, and the federal state of Upper Austria.

{\setlength{\parskip}{0.em}

}

\end{document}